\documentclass[12pt,twoside, a4paper]{amsart}
\usepackage{hyperref}
\usepackage{amsmath,amsthm,amsopn,amssymb,a4wide,varioref}
\usepackage{verbatim}
\usepackage{enumitem}
\usepackage{color, bm, amscd, tikz-cd}
\usepackage{showkeys}
\usepackage{color, bm, amscd, tikz-cd}

\usepackage{centernot}
\usepackage{mathtools}
\usepackage{stmaryrd}
\makeatletter
\newcommand{\xMapsto}[2][]{\ext@arrow 0599{\Mapstofill@}{#1}{#2}}
\def\Mapstofill@{\arrowfill@{\Mapstochar\Relbar}\Relbar\Rightarrow}
\makeatother

\newcommand{\cB}{{\mathcal B}}

\newcommand{\cD}{{\mathcal D}}
\newcommand{\cE}{{\mathcal E}}
\newcommand{\cF}{{\mathcal F}}
\newcommand{\cG}{{\mathcal G}}
\newcommand{\cH}{{\mathcal H}}

\newcommand{\cK}{{\mathcal K}}

\newcommand{\cO}{{\mathcal O}}

\newcommand{\cQ}{{\mathcal Q}}

\newcommand{\bV} {{\textbf{V}}}
\newcommand{\bcK}{{\boldsymbol{\mathcal K}}}
\newcommand{\bPi}{{\boldsymbol{\Pi}}}
\newcommand{\bT}{{\mathbb{\mathbb T}}}

\newcommand{\bD}{{\mathbb D}}

\newcommand{\oq}{{\overline{q}}}
\newcommand{\bev}{\mathbf{ev}}

\newcommand\restr[2]{\ensuremath{\left.#1\right|_{#2}}}

\newcommand{\sbm}[1]{\left[\begin{smallmatrix} #1
		\end{smallmatrix}\right]}

\newtheorem{thm}{Theorem}[section]

\newtheorem{corollary}[thm]{Corollary}
\newtheorem{lemma}[thm]{Lemma}

\newtheorem{proposition}[thm]{Proposition}

\theoremstyle{definition}

\newtheorem{definition}[thm]{Definition}
\newtheorem{remark}[thm]{Remark}

\newtheorem{example}[thm]{Example}
\numberwithin{equation}{section}

\def\textmatrix#1&#2\\#3&#4\\{\bigl({#1 \atop #3}\ {#2 \atop #4}\bigr)}
\def\dispmatrix#1&#2\\#3&#4\\{\left({#1 \atop #3}\ {#2 \atop #4}\right)}
\numberwithin{equation}{section}

\def\textmatrix#1&#2\\#3&#4\\{\bigl({#1 \atop #3}\ {#2 \atop #4}\bigr)}
\def\dispmatrix#1&#2\\#3&#4\\{\left({#1 \atop #3}\ {#2 \atop #4}\right)}

\begin{document}
	\title[Dilation and Model Theory]{Dilation and Model Theory for Pairs of Contractions with a Twisted Commutation Relation}

		\author[Sourav Ghosh]{Sourav Ghosh}
	\address[]{Department of Mathematics, Indian Institute of Science Education and Research Pune, Dr.\ Homi Bhabha Road, Maharashtra 411008, India.}
	\email{sourav.ghosh@students.iiserpune.ac.in}
	
	%\author[Sau]{Haripada Sau}
	%\address[Sau]{Department of Mathematics, Indian Institute of Science Education and Research, Dr.\ Homi Bhabha Road, Pashan, Pune, Maharashtra 411008, India.}
	%\email{haripadasau215@gmail.com, hsau@iiserpune.ac.in}


	\begin{abstract}
	In this note, we develop a parallel theory of the classical Sz.-Nagy--Foias dilation and model theory for a single contraction operator in the setting of pairs of \em{{$q$-commuting}} contraction operators for a unimodular complex number $q$. 
	\end{abstract}
	
	\maketitle
	%\tableofcontents
	
	\section{Introduction}
	A celebrated theorem by Sz.-Nagy \cite{sz-nagy} states that every contraction on a Hilbert space admits a co-isometric extension. More precisely, for any contraction \( T \) acting on a Hilbert space \( \mathcal{H} \), there exists an isometry \( V \) acting on a Hilbert space \( \mathcal{K} \) that contains \( \mathcal{H} \), such that \( T^* = V^*|_{\mathcal{H}} \). A similar result for a pair of commuting contractions was provided in a seminal paper by And\^o \cite{ando}, which shows that any pair \( (T_1, T_2) \) of commuting contractions on a Hilbert space \( \mathcal{H} \) has a commuting isometric pair \( (V_1, V_2) \) on a Hilbert space \( \cK \) containing \( \cH \) as a co-isometric extension. Although the extension of And\^o's theorem to more than two variables is no longer valid (see \cite{Parrott}), it has inspired numerous generalizations across various frameworks. This note is concerned with the extension of And\^o's Theorem in the non-commutative setting. Sebesty\'en \cite{Sebestyen} obtained the first non-commutative generalization of And\^o's theorem, focusing on the anti-commutative case. He showed that if $(T_1,T_2)$ is a contractive pair which is \textit{anti-commuting}, i.e., $T_1T_2=-T_2T_1$, then there is an anti-commuting isometric lift $(V_1,V_2)$ of $(T_1,T_2)$. Extending Sebesty\'en's work, Keshari and Mallick showed in \cite{KM2019} that if $(T_1,T_2)$ is a contractive pair which is \emph{$q$-commuting} i.e., $T_1T_2=qT_2T_1$ for some unimodular complex number \(q\), then there is a $q$-commuting isometric lift $(V_1,V_2)$ of $(T_1,T_2)$. However, these dilations are far from canonical, a fact that seems to rule out a coherent model theory. The methods in the works cited above are based on arguments leading to an existential proof. 
	
The main contributions of this paper are as follows:
\begin{enumerate}
\item We give two constructive proofs of the lifting theorem by Keshari and Mallick with an emphasis on the function-theoretic interpretation. See the forward directions of Theorem \ref{T:SDil} and Theorem \ref{T:Dmod} for the explicit constructive proofs.
\item We show by a concrete example (see Example \ref{nonisolifts}) that a $q$-commuting contractive pair may have two minimal $q$-commuting isometric lifts which are not unitarily equivalent.
\item We describe \textit{all} $q$-commuting isometric lifts of a given $q$-commuting contractive pair; See the converse directions of Theorem \ref{T:SDil} and Theorem \ref{T:Dmod} for the descriptions in two different models.
\item In what follows, for a contraction operator $T$ acting on a Hilbert space $\cH$, we shall denote by
$$
D_T=(I_\cH-T^*T)^{\frac{1}{2}}\quad\mbox{and}\quad
\cD_T=\overline{\operatorname{Ran}D_T}
$$the \textit{defect operator} and the \textit{defect space} for $T$, respectively. The explicit functional model of the isometric lifts makes it possible to develop a functional model for pairs of \( q \)-commuting contractions. See the results in Section \ref{S:FunctionalModel}. The functional model has three key components: the Sz.-Nagy--Foias \textit{characteristic function} 
$$
\Theta_T(z)=-T+zD_{T^*}(I-zT^*)^{-1}D_T|_{\cD_T}:\cD_T\to\cD_{T^*}
$$for the product contraction $T=T_1T_2=qT_2T_1$, a $q$-commuting unitary pair $(W_1,W_2)$ canonically constructed in Subsection \ref{SS:CanonUniPair} from a $q$-commuting contractive pair, and a novel notion of fundamental operators for $q$-commuting contractive pairs introduced in Subsection \ref{SS:CharcTriple}. These three components constitute what is called a characteristic triple for a $q$-commuting contractive pair. It is shown in Theorem \ref{T:CharcUniInv} that the characteristic triple serves as a complete unitary invariant for $q$-commuting contractive pairs whose product is completely non-unitary (\textit{cnu}, for short).
\item We then introduce a notion of admissible triple $\Xi=((G_1,G_2),(W_1,W_2),\Theta)$ consisting of an operator pair $(G_1,G_2)$, a $q$-commuting unitary pair $(W_1,W_2)$ and a contractive analytic function $\Theta$ so that the characteristic triple for the model $q$-commuting contractive pair $(T_{1,\Xi},T_{2,\Xi})$ corresponding to $\Xi$ coincides with the admissible triple $\Xi$ in an appropriate sense. In this way, we arrive at an exact analogue of Sz.-Nagy--Foias model theory for contractive operators in the case of $q$-commuting contractive pairs.
\end{enumerate}
There has been considerable interest in \(q\)-commuting operators from various perspectives within operator theory; see, for example, \cite{BPS}, \cite{GS}, \cite{Tomar}. For further developments and generalizations to \(Q\)-commuting contractions, where \(Q\) is a unitary operator, we refer the reader to \cite{BB}, \cite{MS}, and the references therein. %Interested readers may also look at the work of Opěla \cite{Opela}.	
	
	%%%%%%%%%%%%%%%%%%%%%%%%%%%%%%%%%%%%%%%%%%%%%%%%%%%%%%%%%
	
	\section{Preliminaries}
	We begin with the preliminaries. %For a contractive Hilbert space operator $T$, the following notations for the \textit{defect operator} and the \textit{defect space} are standard and will be frequently used:
	%\begin{align*}
	%D_T:=(I-T^*T)^{1/2} \quad\mbox{and}\quad
	%\cD_T=\overline{\operatorname{Ran}D_T}.
	%\end{align*}
	Let \( (T_1, T_2) \) be a pair of \( q \)-commuting contractions on a Hilbert space \( \cH \), and let \( T = T_1 T_2 = q T_1 T_2 \). Then we have the following relations among the defect operators:
	
	\begin{align} \label{Deft1}
		D_T^2 = I - T^* T &= I - T_2^* T_2 + T_2^* T_2 - T_2^* T_1^* T_1 T_2 = D_{T_2}^2 + T_2^* D_{T_1}^2 T_2 
		\\ \label{Deft2}&= I - T_1^* T_1 + T_1^* T_1 - T_1^* T_2^* T_2 T_1 = D_{T_1}^2 + T_1^* D_{T_2}^2 T_1.
	\end{align}
	
Then the map \( \Lambda: \mathcal{D}_T \to \mathcal{D}_{T_1} \oplus \mathcal{D}_{T_2} \), defined by
	\begin{align}\label{Lambda}
	\Lambda D_T h = D_{T_1} T_2 h \oplus  D_{T_2} h \quad \text{for every } h \in \cH, 
		\end{align}
	is an isometry as
	
	\begin{align*}
		\| \Lambda D_T h \|^2 
		= \| D_{T_1} T_2 h \|^2 + \| D_{T_2} h \|^2 
		&= \| T_2 h \|^2 - \| T_1 T_2 h \|^2 + \| h \|^2 - \| T_2 h \|^2 \\
		&= \| h \|^2 - \| T h \|^2 = \| D_T h \|^2.
	\end{align*}
	
	Moreover, the map \( U: \operatorname{Ran}\Lambda \to \mathcal{D}_{T_1} \oplus \cD_{T_2} \), defined by
	\begin{align}\label{U}
	U(D_{T_1} T_2 h \oplus  D_{T_2} h) =  D_{T_1} h \oplus  D_{T_2} T_1 h \quad \text{for } h\in \cH, 
	\end{align}
	is also an isometry as can be checked by a similar computation. If required, by adding a direct copy of $\cH$ to \( \cD_{T_1} \oplus \cD_{T_2} \), we can extend \( U \) to a unitary operator on \( \cD_{T_1} \oplus \cD_{T_2} \oplus\cH\), still denoted by \( U \). The isometry and unitary operators defined above will play a pivotal role in what follows. We give them a name for reference.
	
	\begin{definition}\label{D:SpecialAndo}
		Given a pair of \(q\)-commuting contractions $(T_1,T_2)$ on a Hilbert space $\cH$, the tuple $(\cF;\Lambda,P,U)$, as constructed above, is called a \textit{special And\^o tuple} for $(T_1,T_2)$, where $\cF=\cD_{T_1}\oplus\cD_{T_2}\oplus\cE$ (with the understanding that $\cE$ may possible be trivial), $\Lambda:\cD_{T}\to\cF$ is the isometry as in \eqref{Lambda}, $U$ acting on $\cF$ is the unitary as defined in \eqref{U}, and $P$ is the orthogonal projection of $\mathcal{F}$ onto the first component, i.e., $\mathcal{D}_{T_1}\oplus\{0\}\oplus\{0\}$.
	\end{definition}
	
	%%%%%%%%%%%%%%%%%%%%%%%%%%%%%%%%%%%%%%%%%%%%%%%%%%%%%%%%%%%%%%%%
	
	\begin{remark} If \((T_1, T_2)\) is a pair of \(q\)-commuting operators then so is \((T_1^*, T_2^*)\). The special And\^o tuple for $(T_1^*,T_2^*)$ will be denoted by \((\mathcal{F}_*; \Lambda_*, P_*, U_*)\), where 
$$\mathcal{F}_*=\mathcal{D}_{T_1^*} \oplus \mathcal{D}_{T_2^*}\quad\mbox{or}\quad
 \cF_*=\mathcal{D}_{T_1^*} \oplus \mathcal{D}_{T_2^*} \oplus \cH,
 $$ \(\Lambda_*: \mathcal{D}_{T^*} \to \cF_*\) is an isometry defined as:
		\[
		\Lambda_* D_{T^*} h = D_{T_1^*} T_2^* h \oplus  D_{T_2^*} h \oplus 0 \quad \mbox{for all } h \in \mathcal{H},
		\]
		\(P_*\) is the orthogonal projection of $\cF_*$ onto the first component \(\mathcal{D}_{T_1^*}\oplus\{0\}\oplus\{0\}\), and with a slight abuse of notation, \(U_*\) is the unitary extension on $\cF_* $ of the map \(U_*: \operatorname{Ran}\Lambda_* \to \cF_*\), explicitly defined as:
		\[
		U_*(D_{T_1^*} T_2^* h \oplus  D_{T_2^*} h\oplus 0) =  D_{T_1^*} h \oplus  D_{T_2^*} T_1^* h\oplus 0 \quad \mbox{for all }h \in \mathcal{H}.
		\]

	\end{remark}
For Hilbert spaces $\cH$ and $\cK$, we shall use the direct sum notations
	$$
\cH\oplus\cK \quad\mbox{and}\quad
\begin{bmatrix}
\cH \\
\cK
\end{bmatrix}	
	$$interchangeably; similar convention is used for an element in the direct sum, i.e., for vectors $h\in\cH$ and $k\in\cK$, the notations $h\oplus k$ and $\sbm{h \\ k}$ mean the same vector in the direct sum of $\cH$ and $\cK$.
	
	We note some properties of special And\^o tuples below which are crucial for later developments. 
	\begin{lemma}
		Let $(T_1,T_2)$ be a $q$-commuting pair of contractions on $\cH$, and $(\cF;\Lambda,P,U)$ and $(\cF_*;\Lambda_*,P_*,U_*)$ be the special And\^o tuples for $(T_1,T_2)$ and $(T_1^*,T_2^*)$, respectively. Then we have 
		
		\begin{align}\label{Prop1} 
			\begin{cases} 
				&T_1^*T_1+D_T\Lambda^*U^*PU\Lambda D_T=I_{\cH}, \quad T_2^*T_2 + D_T\Lambda^*P^\perp \Lambda D_T=I_{\cH}\\ 
				&PU\Lambda D_TT_2+P^\perp \Lambda D_T=\Lambda D_T=U^*(P^\perp \Lambda D_TT_1+PU\Lambda D_T); 
			\end{cases}
		\end{align}
		
		and
		\begin{align}\label{Prop2} 
			\begin{cases} 
				&\Lambda_* D_{T^*}T_1^*=P_*^\perp U_*\Lambda_* D_{T^*}+P_* U_* \Lambda_* D_{T^*}T^* \\ 
				&\Lambda_* D_{T^*}T_2^*=U_*^*P_*\Lambda_* D_{T^*}+q U_*^*P_*^\perp \Lambda_* D_{T^*}T^*. 
			\end{cases} 
		\end{align} 
	\end{lemma}
	%%%%%%%%%%%%%%%%%%%%%%%%%%%%%%%%%%%%%%%%%%%%%%%%%%%%%%%%%
	
	\begin{proof}
		The first equality in \eqref{Prop1} follows from the simple inner product computation for every \( h,h' \in \mathcal{H} \):

		\begin{align*}\langle D_T\Lambda^*U^*PU\Lambda D_Th,h'\rangle=\langle PU \Lambda D_T h, PU \Lambda D_T h' \rangle = \langle D_{T_1} h \oplus 0, D_{T_1} h' \oplus 0 \rangle = \langle D_{T_1}^2 h, h' \rangle.\end{align*}
		The second equality in \eqref{Prop1} follows in a similar manner. For the third set of equalities in \eqref{Prop1}, we note that for every \( h \in \mathcal{H} \),
		
		\[
		\begin{aligned}
			PU \Lambda D_T T_2 h + P^\perp \Lambda D_T h &= PU \begin{bmatrix} D_{T_1} T_2 \\ D_{T_2} \end{bmatrix} T_2 h + P^\perp \begin{bmatrix} D_{T_1} T_2 \\ D_{T_2} \end{bmatrix} h \\
			& = P \begin{bmatrix} D_{T_1} \\ D_{T_2} T_1 \end{bmatrix} T_2 h + \begin{bmatrix} 0 \\ D_{T_2} \end{bmatrix} h \\
			& = \begin{bmatrix} D_{T_1} T_2 \\ D_{T_2} \end{bmatrix} h = \Lambda D_T h.
		\end{aligned}
		\]
		The other equality \( \Lambda D_T = U^* \left( P^\perp \Lambda D_T T_1 + PU \Lambda D_T \right) \) follows from a similar computation.

		%%%%%%%%%%%%%%%%%%%%%%%%%%%%%%%%%%%%%%%%%%%%%%%%%%%%%%%%%%%%%%%%%%%%%
		
		For the set of equations in \eqref{Prop2}, we see that for $h\in\cH$, 
		\begin{eqnarray*} 
			&&(P_*^\perp U_*\Lambda_* D_{T^*}+P_* U_* \Lambda_* D_{T^*}T^*)h\\
			&&= P_*^\perp U_{*}({D_{T_1^*}}T_2^*h\oplus {D_{T_2^*}}h)+P_* U_{*}({D_{T_1^*}}T_2^*T^*h\oplus {D_{T_2^*}}T^*h)\\ 
			&&=(0\oplus D_{T_2^*}T_1^*h)+(D_{T_1^*}T^*h\oplus 0) \\
			&&=D_{T_1^*}T_2^*T_1^*h\oplus D_{T_2^*}T_1^*h\\
			&&=\Lambda_* D_{T^*}T_1^*h. 
		\end{eqnarray*}
The other equality is as follows:
		\begin{eqnarray*}
			&&(U_*^*P_*\Lambda_* D_{T^*}+qU_*^*P_*^\perp \Lambda_*D_{T^*}T^*)h\\
			&&=U^*_*P_*(D_{T^*_1}T^*_2h \oplus D_{T^*_2}h)+qU^*_*P^{\perp}_* (D_{T^*_1}T^*_2T^*h \oplus D_{T^*_2}T^*h)\\
			&&=U^*_*(D_{T^*_1}T^*_2h \oplus 0)+qU^*_*(0\oplus D_{T^*_2}T^*h)\\
			&&=U^*_*(D_{T^*_1}T^*_2h\oplus D_{T^*_2}T^*_1 T^*_2h) \quad [\mbox{since }qT^*=T_1^*T_2^*]\\
			&&=D_{T^*_1}T^*_2T^*_2h\oplus D_{T^*_2} T^*_2h \\
			&&=\Lambda_* D_{T^*}T^*_2h.
		\end{eqnarray*}
		This completes the proof of the lemma.
	\end{proof}

	%%%%%%%%%%%%%%%%%%%%%%%%%%%%%%%%%%%%%%%%%%%%%%%%%%%%%%%%%
	We need another preparatory lemma.
	\begin{lemma}\label{lem:Prod}
		Let $(T_1,T_2)$ be a $q$-commuting pair of contractions on a Hilbert space $\cH$ and $T=T_1T_2=qT_2T_1$.
		\begin{enumerate} 
			\item If $T$ is an isometry, then so are both $T_1$ and $T_2$. The statement remains true even when `isometry' is replaced by `unitary'. 
			\item For each $n\geq1$, $T_1T^n=q^nT^nT_1$ and $T_2T^n=\oq^n T^n T_2$. 
		\end{enumerate} 
	\end{lemma} 
	
	%%%%%%%%%%%%%%%%%%%%%%%%%%%%%%%%%%%%%%%%%%%%%%%%%%%%%%%%%
	\begin{proof}
		
		From equations \eqref{Deft1} and \eqref{Deft2}, we observe that for every vector \( h \),
		
		\[
		\|D_T h\|^2 = \|D_{T_1} T_2 h\|^2 \oplus \|D_{T_2} h\|^2 = \|D_{T_1} h\|^2 + \|D_{T_2} T_1 h\|^2.
		\]
		
This directly proves that if $T=T_1T_2$ is an isometry, then so are $T_1$ and $T_2$. To see the assertion about `unitary', simply note that if \( (T_1, T_2) \) is \( q \)-commuting, then \( (T_1^*, T_2^*) \) is also \( q \)-commuting. This completes the proof of part (1). 

The proof of part (2) is a routine computation using the $q$-commuting relation $T=T_1T_2=qT_2T_1$ by succession.
	\end{proof}
	
	Note that by taking adjoints in part (2), we obtain the relations 
	\begin{align}
		 T_1^* T^{*n} = q^n T^{*n} T_1^* \quad\mbox{and} \quad T_2^* T^{*n} = \oq^n T^{*n} T_2^*.
\end{align} This will be used as well.
	
	%%%%%%%%%%%%%%%%%%%%%%%%%%%%%%%%%%%%%%%%%%%%%%%%%%%%%%%%%
	\subsection{\texorpdfstring{Models for $q$-commuting isometries}{Models for q-commuting isometries}}
	
The	following structure theorem, referred to as the Wold decomposition for isometries, attributed to Halmos \cite{Halmos}, von Neumann \cite{von} and Wold \cite{Wold}, is the stepping stone to the structure theory for Hilbert space operators.
	\begin{thm} \label{T:vnw}
		Let $V$ be an isometry on $\cH$. Then $V$ can be identified with 
		$$ 
		\begin{bmatrix}
		M_z&0\\0&W
		\end{bmatrix}: 
		\begin{bmatrix}
		H^2(\cD_{V^*})\\
		 \bigcap_{n\geq 0}V^n\cH
		\end{bmatrix}
		\to 
		\begin{bmatrix}H^2(\cD_{V^*})\\ \bigcap_{n\geq 0}V^n\cH
		\end{bmatrix}, 
		$$where $M_z:H^2(\cD_{V^*})\to H^2(\cD_{V^*})$ is the shift operator $f\mapsto z f(z)$, and $W=V|_{\bigcap_{n\geq 0}V^n\cH}$ is the unitary operator. The identification above is via the unitary similarity $\tau:\cH\to H^2(\cD_{V^*})\oplus \cH_u$ given as 
		$$ 
		\tau h = \begin{bmatrix}
		D_{V^*}(I-zV^*)^{-1}h \\
		\lim_{n \to \infty} V^n V^{*n} h 
		\end{bmatrix}.$$
	\end{thm}
	There have been a substantial amount of work done in order to find a model for tuples of commuting isometries. Perhaps the most prominent models is the one developed by Berger, Coburn, and Lebow in \cite{BCL}. A similar model for $q$-commuting pair of isometries is recently found in \cite{BS-qBCL} by J. Ball and H. Sau. We state the model for $q$-commuting setting. The substitution $q=1$ in the model theorem below yields the model found by Berger, Coburn, and Lebow for the commuting case.
	\begin{thm}\label{T:qBCL}
		Let $V_1$ and $V_2$ be two operators acting on a Hilbert space $\cH$ and $V=V_1V_2$. Then the following are equivalent:
 \begin{enumerate}
 \item  The pair $(V_1,V_2)$ is $q$-commuting.
 
 \item  There exist Hilbert spaces $\cF$ and $\cK_u$, a projection $P$ and a unitary $U$ in $\cB(\cF)$, 
 and a pair $(W_1,W_2)$ of $q$-commuting 
 unitaries in $\cB(\cK_u)$ such that $(V_1,V_2)$ is unitarily equivalent to
\begin{align}  \label{isomod1}
\left(
\begin{bmatrix} (R_q \otimes P^\perp U) + (M_z R_q \otimes P U) & 0  \\  0& W_1 \end{bmatrix},
\begin{bmatrix} (R_\oq \otimes U^*P)+ (R_\oq M_z \otimes U^*P^\perp) & 0\\0&W_2 \end{bmatrix}
\right)
\end{align}
acting on $\begin{bmatrix} H^2\otimes \cF\\\cK_u \end{bmatrix},$ where $R_q$ is the \textit{rotation operator} on $H^2$ defined by $$R_q(f)(z)=f(qz).$$
%Moreover, the tuple $(\cF,\cK_u;P,U,W_1,W_2)$ can be chosen to be such that
%\begin{align}\label{ExplctTuple1}
%\begin{cases}
%&\cF=
%\begin{bmatrix}
%\cD_{V_1^*}\\  \cD_{V_2^*}
%\end{bmatrix}
%, \quad \cK_u=\bigcap_{n\geq 0}V^n\cH,\quad P=\begin{bmatrix}
%I_{\cD_{V_1^*}}&0\\0&0
%\end{bmatrix},  \\
%& U: \begin{bmatrix} D_{V_1^*}\\D_{V_2^*}V_1^* \end{bmatrix} h \mapsto 
%\begin{bmatrix} D_{V_1^*}V_2^*\\D_{V_2^*} \end{bmatrix} h \text{ for } h \in \cH
%\text{ and } (W_1,W_2)=(V_1,V_2)|_{\cK_u},
%\end{cases}
% \end{align}
% and the unitary operator $\tau:\cH\to \begin{bmatrix} H^2\otimes \cF\\\cK_u  \end{bmatrix}$ 
% can be chosen to be

 \item  There exist Hilbert spaces $\cF_\dag$ and $\cK_{u\dag}$, a projection $P_\dag$ and a unitary $U_\dag$ in 
 $\cB(\cF_\dag)$, 
 and a pair $(W_{1\dag},W_{2\dag})$ of $q$-commuting unitaries in $\cB(\cK_{u\dag})$ such that $(V_1,V_2)$ is unitarily equivalent to the pair
\begin{align}  \label{qiso-model2}
\begin{bmatrix}
R_q \otimes U_\dag^*P_\dag^\perp + M_z R_q \otimes U_\dag^*P_\dag & 0\\ 0 & W_{1\dag} \end{bmatrix}, 
\begin{bmatrix}
R_\oq \otimes P_\dag U_\dag + R_\oq M_z \otimes P_\dag^\perp U_\dag&0\\0&W_{2\dag}
\end{bmatrix}
\end{align}
acting on $\begin{bmatrix} H^2\otimes \cF_\dag\\\cK_{u\dag}\end{bmatrix}$. 
\end{enumerate}
Moreover, in both items (2) and (3), the spaces can be chosen so that $\cF=\cF_\dag=\cD_{V^*}$, $\cK_u=\cK_{u\dag}=\cap_{n\geq0}V^n\cH$, and the unitary identification map can be taken to be
 \begin{align}\label{BCL-tau}
\tau h= 
 \begin{bmatrix}
D_{V^*}(I_\cH-zV^*)^{-1}\\
\lim_n V^nV^{*n} 
\end{bmatrix}h,
 \end{align}where $V=V_1V_2$.
	\end{thm}
	Note that the unitary identification map $\tau$ as in \eqref{BCL-tau} is exactly the unitary identification map behind the Wold decomposition for the isometry $V=V_1V_2$ (see Theorem \ref{T:vnw}).
	
The explicit functional model above makes it possible to jointly extend $q$-commuting isometric pairs to $q$-commuting unitary pairs with an additional structure as demonstrated in the result below. This will be used in what follows.
	
	\begin{thm}\label{isoext}
		Every $q$-commuting isometric pair $(X_1, X_2)$ has a $q$-commuting unitary extension $(Y_1, Y_2)$ such that $Y =Y_1Y_2$ is the minimal unitary extension of $X =X_1X_2$.
	\end{thm}
	\begin{proof}
		
		Without loss of generality, suppose the pair of \(q\)-commuting isometries \((X_1, X_2)\) is given by the model obtained in \eqref{isomod1}. Consider the pair \((Y_1, Y_2)\) defined as
		
		\(
		\left( \begin{bmatrix}
			R_q \otimes P^\perp U + M_{\zeta} R_q \otimes P U & 0 \\
			0 & W_1
		\end{bmatrix}, 
		\quad
		\begin{bmatrix}
			R_{\oq} \otimes U^* P + R_{\oq} M_{\zeta} \otimes U^* P^\perp & 0 \\
			0 & W_2
		\end{bmatrix} \right) \\
		\text{on} 
		\begin{bmatrix}
			L^2 \otimes \cF \\ 
			\cK_u
		\end{bmatrix},
				\) where $M_\zeta:L^2\to L^2$ is the unitary operator $f\mapsto \zeta f(\zeta)$.
		
		It is straightforward to verify that \((Y_1, Y_2)\) forms a pair of \(q\)-commuting unitary operators. By considering the natural embedding of \((H^2 \otimes \cF) \oplus \cK_u\) into \((L^2 \otimes \cF) \oplus \cK_u\) through the map
		
		\[
		\begin{bmatrix}
			z^n \otimes \xi \\
			\eta
		\end{bmatrix}
		\mapsto 
		\begin{bmatrix}
			\zeta^n \otimes \xi \\
			\eta
		\end{bmatrix}
		\quad \text{for} \quad \xi \in \cF, \; \eta \in \cK_u, \; n \geq 0,
		\] we see that with $(Y_1,Y_2)$ as set above,
		$$
		Y=Y_1Y_2=\begin{bmatrix}
		M_\zeta&0\\
		0& W_1W_2
		\end{bmatrix}:\begin{bmatrix}
			L^2 \otimes \cF \\ 
			\cK_u
		\end{bmatrix}\to
		\begin{bmatrix}
			L^2 \otimes \cF \\ 
			\cK_u
		\end{bmatrix}
		$$  is the minimal unitary extension of the isometry 
		$$
		X=X_1X_2=\begin{bmatrix}
		M_z&0\\
		0& W_1W_2
		\end{bmatrix}:
		\begin{bmatrix}
			H^2 \otimes \cF \\ 
			\cK_u
		\end{bmatrix}\to
		\begin{bmatrix}
			H^2 \otimes \cF \\ 
			\cK_u
		\end{bmatrix}.
		$$This was to be proved.
	\end{proof}
	
	%%%%%%%%%%%%%%%%%%%%%%%%%%%%%%%%%%%%%%%%%%%%%%%%%%%%%%%%%
	
A well-known result in operator theory is that any bounded operator that commutes with the shift operator $M_z$ is of the form $M_\varphi$ for some bounded analytic function $\varphi$. We conclude this section with an analogous result in the $q$-commuting setting. This will be used every now and then in what follows. It is convenient to have a definition first.
	\begin{definition}
		Let $T$ be a bounded operator on $\cH$ and $q$ be a unimodular complex number. The  {\em{ $q$-commutant}} of $T$, denoted by $\{T\}'_q$, is defined by 
		$$
		\{T\}'_q=\{A\in \cB(\cH):AT=qTA\}.
		$$
	\end{definition}

\begin{proposition}\label{q-commutant}
		Let $M_z$ be the shift operator on the vector valued Hardy space $H^2(\cE)$, for some Hilbert space \(\cE\). Then 
		$$
		\{M_z\}'_q=\{M_{\varphi}R_q: \varphi\in H^{\infty}(\mathcal{B}(\cE))\}.
		$$ 
\end{proposition}
		\begin{proof}
It is easy to see that operators of the form $M_\varphi R_q$ $q$-commutes with $M_z$ for any $\varphi$ in $H^{\infty}(\mathcal{B}(\cE))$.
			
To prove the converse, we shall use the short-hand notation 
$$
\widetilde f(z):= f(qz) \quad\mbox{for }f\in H^2(\cE).
$$Suppose $A$ is a bounded operator on $H^2(\cE)$ so that $AM_z=qM_zA$. It readily implies $AM_{f(z)}=M_{f(qz)}A$ for all polynomials $f\in \mathbb{C}[z]$. Let us denote $\widetilde{f}(z)=f(qz)$. Given any $e\in \cE$, define $\varphi$ on $\mathbb{D}$ by 
\[
\varphi(z)(e):=A(1\otimes e)(z).
\]Let $k$ denote the reproducing kernel (also known as the Sz\"ego kernel) 
$$
k(z,w)=\frac{1}{1-\overline{w}z}$$
for the classical Hardy space $H^2$. It is well-known that for the vector-valued Hardy space $H^2(\cE)$, the reproducing kernel is given by $k\otimes I_\cE$ and that it has the reproducing property
\begin{align}\label{Reproduce}
\langle f, k_z\otimes e\rangle_{H^2(\cE)} = \langle f(z),e\rangle_\cE
\end{align}for every $f\in H^2(\cE)$ and $e\in\cE$. Let us compute for every $e_1,e_2\in\cE$,
$$
\langle \varphi(z)e_1,e_2\rangle_\cE= \langle A(1\otimes e_1)(z),e_2\rangle_\cE= \langle A(1\otimes e_1),k_z\otimes e_2\rangle_{H^2(\cE)}.
$$From this realization, it follows that $\varphi(z)$ is a bounded operator for every $z\in\bD$ and the bound is given by $\|A\|\|k_z\|$. It also follows that $z\mapsto\varphi(z)$ is a $\mathcal{B}(\cE)$-valued holomorphic map. Indeed, it is enough to show that $z\mapsto\varphi(z)e$ is holomorphic for every $e\in\cE$. This can be seen from the very definition of $\varphi$ since \(A(1\otimes e)\in H^2(\cE)\) and hence is holomorphic. Let us also note that for every polynomial $f\in \mathbb{C}[z]$,
$$A(f(z)e)=AM_f(1\otimes  e)(z)=M_{\widetilde{f}}A( 1\otimes e)(z)=\widetilde{f}(z)\varphi(z)(e)=\varphi(z)(\widetilde{f}(z)e).
$$
Given any $f\in H^2(\cE)$, there exists a sequence of \(\cE\)-valued polynomials $(f_n)$ such that $f_n\to f$ in $H^2(\cE)$ as polynomials are dense in $H^2$. Since $A$ is a  bounded linear operator, it follows that $Af_n$ converges to $Af$ in $H^2(\cE)$.  We know that the norm convergence in $H^2(\cE)$ implies the uniform convergence on compact subsets of $\bD$ and hence, in particular, we have $f_n(z)\to f(z)$ and $\widetilde{f_n}(z)=f_n(qz)\to \widetilde{f}(z)=f(qz)$ in \(\cE\) for all \(z\in \mathbb D\). Then $\varphi(z)(\widetilde{f_n}(z)) \to \varphi (z)( \widetilde{f}(z))$ for all $z\in \mathbb{D}$ as the operator $\varphi(z$) is bounded for all $z\in \mathbb{D}$. Also since $ \varphi(z)(\widetilde{f_n}(z))=A(f_n)(z)\to Af(z)$ in $\cE$ for all \(z\in \mathbb{D}\), we conclude that
$Af(z)=\varphi (z) \widetilde{f}(z)=\varphi(z)f(qz)$ for all $z\in \mathbb{D}$. 

Thus what remains to show is that $\|\varphi(z)\|$ is bounded by a constant independent of $z$ in $\bD$. To this end, we apply the conclusion of the preceding paragraph to the choice $f(\cdot)=\widetilde k_{qw}(\cdot)\otimes e$, where $e\in\cE$ i.e.,
$$
A(k_{qw}\otimes e)(z)= \varphi(z)(\widetilde k_{qw}\otimes e)(z).
$$In particular, by substituting  $w=z$ in the above identity of vectors in $\cE$, we obtain the following: 
$$
A( k_{qz}\otimes e)(z)= \varphi(z)(\widetilde k_{qz}\otimes e)(z).
$$We shall use this in the following computation. Since $\widetilde k_{qz}(z)$ is real number, we see that for every $z\in\bD$, and $e_1,e_2\in\cE$,
\begin{align*}
\widetilde k_{qz}(z) |\langle \varphi(z)e_1,e_2 \rangle_{\cE}| &=
|\langle \varphi(z)(\widetilde k_{qz}\otimes e_1)(z),e_2 \rangle_{\cE}|\\
&=|\langle A(k_{qz}\otimes e_1),k_z\otimes e_2 \rangle_{H^2(\cE)}|\le \|A\|\|e_1\|\|e_2\|\|k_z\|^2.
\end{align*}In the computation above, we used the reproducing property \eqref{Reproduce} for the functions $\widetilde k_{qz}$ and the fact that $\|k_{qz}\|_{H^2}=\|k_z\|_{H^2}$. Finally, since $\widetilde k_{qz}(z)=\|k_{qz}\|^2=\|k_z\|^2$, it follows from the computation above that $\|\varphi(z)\|\leq \|A\|$ for every $z\in\bD$. Thus $M_\varphi$ is a bounded operator on $H^2(\cE)$ and therefore we can write $A=M_\varphi R_q$.
\end{proof}

Finally, we shall need the following well known lemma every now and then. We refer the readers to \cite[Lemma 2.13]{BhSrPal} (also see \cite[Lemma 2.3.1]{BS-CUP}) for a proof.
\begin{lemma}\label{L:Zero}
The only bounded linear operator intertwining a unitary
operator with a shift operator is the zero operator, i.e., if $\cK$ and $\cK'$ are Hilbert spaces, $U\in\cB(\cK)$ is a unitary operator and $S\in\cB(\cK')$ is a shift operator, i.e., $S$ is an isometry and $S^{*n}\to 0$ strongly, and
$G: \cK\to\cK'$ is a bounded linear operator such that $GU=SG$, then $G=0$.
\end{lemma}

%%%%%%%%%%%%%%%%%%%%%%%%%
\begin{comment}

Then for all \(w\in \mathbb{D}\) and \(e\in \cE\),  the equality  \[M^*_{\varphi} (k_w\otimes e)= k_w\otimes {\varphi(w)}^* e\] holds as shown by the following calculation :

\begin{align*}
\langle f, M^*_{\varphi} (k_w\otimes e)\rangle=\langle M_{\varphi}f, k_w\otimes e\rangle=\langle \varphi f, k_w\otimes e\rangle=\langle \varphi(w) f(w), e\rangle&=\langle  f(w), {\varphi(w)}^*e\rangle\\&=\langle f, k_w\otimes {\varphi(w)}^* e\rangle.
\end{align*}

Thus we have the following norm inequality:
\[
\|k_w\|\|{\varphi(w)}^* e\|=\|k_w\otimes {\varphi(w)}^* e\|=\|M^*_{\varphi} (k_w\otimes e)\|\le \|M_{\varphi}\| \|k_w\|\| e\|.
\]
Hence \(\|{\varphi(w)}^* e\|\le \|M_{\varphi}\|\| e\|  \) for all \(e\in \cE\)  implies \(\|\varphi(w)\|\le \|M_{\varphi}\| \). Then the boundedness of the operator \(AR_{\bar{q}}=M_{\varphi}\) completes the proof.

	\end{comment}

%%%%%%%%%%%%%%%%%%%%%%%%%%%%%%%%%%%%%%%%%%%%%%%%%%%%%%%%%%%%%%%%%%%%%%%%%%%%%%%%%%%%%%%%

	\section{\texorpdfstring{A Sch\"affer-type model for isometric lifts}{A Sch\"affer-type model for isometric lifts}}
Given a vector-valued Hardy space $H^2(\cF)$, we shall have use of the `evaluation at zero' operator $\bev_0:H^2(\cF)\to\cF$ defined by
$$
\bev_0: f\mapsto f(0).
$$
	
In this section, we start with a $q$-commuting contractive pair $(T_1,T_2)$ and construct a $q$-commuting isometric lift of $(T_1,T_2)$ in a model form that is much similar to the Sch\"affer isometric lift
$$
\begin{bmatrix} T & 0 \\ \bev_0^* D_T & M_z  \end{bmatrix}:
\begin{bmatrix}
\cH \\
H^2(\cD_T)
\end{bmatrix}\to
\begin{bmatrix}
\cH \\
H^2(\cD_T)
\end{bmatrix}
$$for a contraction $T$ acting on $\cH$ constructed in \cite{Schaffer}. The function model Theorem \ref{T:qBCL} for $q$-commuting isometries makes it possible to also describe all isometric lifts of a given $q$-commuting contractive pair.
	\begin{thm} \label{T:SDil}
		Let \((T_1, T_2)\) be a pair of \(q\)-commuting contractions acting on a Hilbert space \(\mathcal{H}\). Let \((\mathcal{F}; \Lambda, P, U)\) be a special And\^o tuple associated with the pair \((T_1, T_2)\) as in Definition \ref{D:SpecialAndo}. Consider the pair of operators \((V^S_1, V^S_2)\) defined on \(\sbm{\mathcal{H} \\ H^2(\mathcal{F}}\) as follows:
	
		\begin{align}\label{SDil}\left( \begin{bmatrix} T_1 & 0 \\ \bev_0^* P U \Lambda D_T &  qM_{(P^\perp + zP)U} R_q \end{bmatrix}, \begin{bmatrix} T_2 & 0 \\ \oq \bev_0^* U^* P^\perp \Lambda D_T &  \oq R_{\oq} M_{U^*(P + z P^\perp)} \end{bmatrix}\right).
		\end{align}
Then $(V^S_1,V^S_2)$ is a $q$-commuting isometric dilation of $(T_1,T_2)$.

Conversely, if $(V_1,V_2)$ is a minimal $q$-commuting isometric lift of $(T_1,T_2)$ acting on $\cK$ containing $\cH$ as a subspace, then there is an And\^o tuple $(\cF,\Lambda,P,U)$ of $(T_1,T_2)$ such that $(V_1,V_2)$ is unitarily equivalent to a pair as in \eqref{SDil}.
%\begin{align}\label{SDilGen}\left( \begin{bmatrix} T_1 & 0 \\ \bev_0^* P U \Lambda D_T &  M_{(P^\perp + zP)U} R_q \end{bmatrix}, \begin{bmatrix} T_2 & 0 \\ \oq \bev_0^* U^* P^\perp \Lambda D_T &  R_{\oq} M_{U^*(P + z P^\perp)} \end{bmatrix} \right)
%		\end{align}acting on $\sbm{\cH \\ H^2(\cF)}$. 
Moreover, the operators $P,U$ and $\Lambda$ satisfy the relations \eqref{Prop1}.
	\end{thm}

	\begin{proof}

		We begin by showing that \( (V^S_1, V^S_2) \) is \( q \)-commuting. To do this, we use the representation \eqref{SDil} to compute
		
		\begin{align*}
			V^S_1 V^S_2 &= \begin{bmatrix}
				T_1 T_2 & 0 \\
				\bev_0^* P U \Lambda D_T T_2 +M_{P^\perp U+ z P U} R_q\bev_0^* U^* P^\perp \Lambda D_T & M_z
			\end{bmatrix}\\
			&= \begin{bmatrix}
				 T_1T_2 & 0 \\
				\bev_0^* (P U \Lambda D_T T_2 + P^\perp \Lambda D_T) & M_z
			\end{bmatrix}
		\end{align*} and 
		\begin{align*} V^S_2 V^S_1 &= \begin{bmatrix}
				T_2 T_1 & 0 \\
				\oq \bev_0^* U^* P^\perp \Lambda D_T T_1 + \oq R_{\oq} M_{U^*(P + z P^\perp)} \bev_0^* P U \Lambda D_T & R_{\oq} M_z R_q
			\end{bmatrix}
			\\&= \begin{bmatrix}
				\oq T_1T_2 & 0 \\
				\oq \bev_0^* U^* (P^\perp \Lambda D_T T_1 + P U \Lambda D_T) & \oq M_z
			\end{bmatrix}.
		\end{align*}

		Therefore, based on the third set of equalities in \eqref{Prop1}, we can say that the pair \( (V^S_1, V^S_2) \) is \( q \)-commuting.

		It remains to show that \(V^S_1 \) and \(V^S_2 \) are isometric operators. We observe that
		
		\begin{align*}
			V^{S*}_1 V^S_1& = \begin{bmatrix}
				T_1^* & D_T \Lambda^* U^* P \bev_0 \\
				0 &\oq  R_{\oq} M_{(P^\perp + zP) U}^* R_q
			\end{bmatrix}
			\begin{bmatrix}
				T_1 & 0 \\
				\bev_0^* P U \Lambda D_T & q M_{(P^\perp + zP) U} R_q
			\end{bmatrix} \\
			&= \begin{bmatrix}
				T_1^* T_1 + D_T \Lambda^* U^* P U \Lambda D_T & 0 \\
				0 & R_{\oq} M_{(P^\perp + zP) U}^* M_{(P^\perp + zP) U} R_q
			\end{bmatrix}
			= \begin{bmatrix}
				I_{\mathcal{H}} & 0 \\
				0 & I_{H^2(\mathcal{F})}
			\end{bmatrix}.
		\end{align*}
		
		The final matrix equality in the above computation is derived by applying the first equation in \eqref{Prop1}. In a similar manner, applying the second equation in \eqref{Prop1} gives \( V_2^S \) an isometry as well.
	
We now prove the converse direction. Since $(V_1,V_2)$ acting on $\cK$ is a lift of $(T_1,T_2)$ acting on $\cH$ and $\cK$ contains $\cH$ as a subspace, we must have
$$
(V_1,V_2)=
\left(
\begin{bmatrix}
T_1&0\\
C_1&D_1
\end{bmatrix},
\begin{bmatrix}
T_2&0\\
C_2&D_2
\end{bmatrix}
\right) \mbox{ acting on } 
\begin{bmatrix}
\cH \\
\cK\ominus\cH
\end{bmatrix}.
$$Since $(V_1,V_2)$ is $q$-commuting isometric pair, $(D_1,D_2)$ is also a $q$-commuting isometric pair. To represent the model in the desired form, we choose to work with the  $q$-commuting pair $(\oq D_1,qD_2)$. By part (2) of Theorem \ref{T:qBCL}, $(\oq D_1,q D_2)$ is jointly unitarily equivalent to
\begin{align*}
\left(
\begin{bmatrix} (R_q \otimes P^\perp U) + (M_z R_q \otimes P U) & 0  \\  0& W_1 \end{bmatrix},
\begin{bmatrix} (R_\oq \otimes U^*P)+ (R_\oq M_z \otimes U^*P^\perp) & 0\\0&W_2 \end{bmatrix}
\right).
\end{align*} Since our goal is to find a unitary copy of the original isometric lift, we assume that $(\oq D_1,qD_2)$ is exactly equal to the pair above. Thus we are assuming that $\cK\ominus\cH = \sbm{H^2(\cF) \\ \cG}$. Suppose $C_j:\cH\to \sbm{H^2(\cF) \\ \cG}$ are given by the matrix forms
$$
C_j=\begin{bmatrix}
C_{j1} \\
C_{j2}
\end{bmatrix}
$$for each $j=1,2$. Since $W_1,W_2$ are unitary operators and $V_1,V_2$ are only contractive operators, we must have $C_{j2}=0$ for each $j=1,2$. Thus we have the original isometric lift $(V_1,V_2)$ unitarily equivalent to the pair
\[
\resizebox{\textwidth}{!}{$
\left(
\begin{bmatrix}
T_1 & 0 & 0\\
C_{11} & q \big[ (R_q \otimes P^\perp U) + (M_z R_q \otimes P U) \big] & 0  \\  
0 & 0 & q W_1
\end{bmatrix}
\begin{bmatrix}
T_2 & 0 & 0\\
C_{21} & \oq \big[ (R_\oq \otimes U^* P) + (R_\oq M_z \otimes U^* P^\perp) \big] & 0\\
0 & 0 & \oq W_2
\end{bmatrix}
\right)
$}
\] acting on $\sbm{\cH \\ H^2(\cF)\\ \cG}$.  Let us now use the fact that $(V_1,V_2)$ is a minimal lift, i.e.,
$$
\begin{bmatrix}
\cH\\
H^2(\cF)\\
\cG
\end{bmatrix}=
\bigvee_{n_1n_2\geq0}
V_1^{n_1}V_2^{n_2}\begin{bmatrix}
h\\
0\\
0
\end{bmatrix}.
$$Given the $3\times3$ block matrix representation of the pair $(V_1,V_2)$, minimality yields $\cG=\{0\}$. Consequently, the minimal isometric lift $(V_1,V_2)$ is unitarily equivalent to
$$
\left(\begin{bmatrix}
T_1&0\\
C_{11}&q[(R_q \otimes P^\perp U) + (M_z R_q \otimes P U)]
\end{bmatrix}
\begin{bmatrix}
T_2&0\\
C_{21}&\oq[(R_\oq \otimes U^*P)+ (R_\oq M_z \otimes U^*P^\perp)]
\end{bmatrix}
\right)
$$acting on $\sbm{\cH \\ H^2(\cF)}$. It remains to argue that the operators $C_{11}$ and $C_{21}$ are as stated. Let us note that
$$
V=V_1V_2=\begin{bmatrix}
T_1T_2&0\\
C& M_z
\end{bmatrix}
$$for some operator $C:\cH\to H^2(\cF)$. Since $V$ is an isometry, we have
$$
C^*C=I_\cH-T^*T \quad\mbox{and}\quad M_z^*C=0.
$$These two identities together gives
$$
C=\bev_0^*\Lambda D_T
$$for some isometry $\Lambda:\cD_T\to\cF$. Since $V_1,V_2$ are isometries as well, we also have
$$
V_1=V_2^*V_2V_1=\overline{q}V_2^*V \quad\mbox{and}\quad
V_2=V_1^*V.
$$Equating the (21) entries in $V_1=\overline{q}V_2^* V$ gives
\begin{align*}
C_{11}=\left((R_q \otimes PU)+ (M_z^* R_q\otimes P^\perp U)\right)\bev_0^*\Lambda D_T =  \bev_0^* P U \Lambda D_T .
\end{align*}
Similar computation with the identity $V_2=V_1^*V$ gives $C_{21}=\oq\bev_0^*  U^*P^\perp \Lambda D_T$ .
\end{proof}

	%%%%%%%%%%%%%%%%%%%%%%%%%%%%%%%%%%%%%%%%%%%%%%%%%%%%%%%%%
	\section{A Douglas-type model for isometric lifts}\label{D:Mod}
	
In this section we develop another functional model for isometric lifts of a \(q\)-commuting contractive pair. The model presented here is arguably more elegant than the one presented in Theorem \ref{T:SDil}. Indeed, the functional model presented in this section is exactly in the model form presented in Theorem \ref{T:qBCL} and consequently, the isometric embedding of the abstract Hilbert space $\cH$ into the model space is a non-trivial isometric embedding unlike the inclusion map in the model presented in Theorem \ref{T:SDil}. Moreover, the model presented in this section will give a direct passage to further development of model theory for $q$-commuting contractive pairs as we shall see in the next section.

First, we associate a pair \((W_1, W_2)\) of \(q\)-commuting unitaries to a $q$-commuting contractive pair and show that this association is canonical. To achieve this, we first recall briefly the construction of an isometric lift of a single contraction operator due to Douglas \cite{Doug-Dilation}. So let \(T \) be a contraction acting on a Hilbert space $\cH$. Then the sequence \(\{T^n T^{*n}\}\) of positive semi-definite operators converges in \(\mathcal{B(H)}\) with respect to the Strong Operator Topology ($\operatorname{SOT}$). Define \(Q_{T^*}\) as the positive square root of the limit operator, that is, 
	\begin{align}\label{Q}
		Q_{T^*}^2 := \lim T^n T^{*n}.
	\end{align}
From the definition of $Q_{T^*}$ it follows that $TQ_{T^*}^2T^*=Q_{T^*}^2$, and so the operator  $X^*: \overline{\operatorname{Ran}Q_{T^*}}\to \overline{\operatorname{Ran}Q_{T^*}}$ defined densely by $$X^*Q_{T^*}h=Q_{T^*}T^*h \quad \text{for} \quad h\in \cH $$ is an isometry. Let $W^*_D$ on $\cQ_{T^*}\supseteq \overline{\operatorname{Ran} Q_{T^*}}$  be the minimal unitary extension of $X^*$, i.e., the space $\cQ_{T^*}$ is given by
	\begin{align}\label{QT-star}
	\cQ_{T^*}=\bigvee_{n\ge0}\{W_D^n\overline{\operatorname{Ran} Q_{T^*}}\}.
	\end{align} Consider next the operator $\cO_{D_{T^*},T^*} :\mathcal{H}\to H^2(\mathcal{D}_{T^*})$ defined by 
	\begin{align}\label{opQ}
	\cO_{D_{T^*},T^*}h=\sum_{n\ge 0}z^n D_{T^*} T^{*n}h=D_{T^*}(I_\cH-zT^*)^{-1}h \quad \text{for}\quad  h\in \cH.\end{align}
	\noindent
	It can be easily seen that the operator $\Pi_D:\mathcal{H}\to H^2(\mathcal{D}_{T^*})\oplus \cQ_{T^*}$ defined by 
	\begin{align}\label{1-D-PiD}
	\Pi_D(h)=
	\begin{bmatrix}
	\cO_{D_{T^*},T^*}h\\ Q_{T^*}
	\end{bmatrix}h
	\end{align}
is an isometry with the intertwining property 
$$
\Pi_D T^*=V_D^* \Pi_D .
$$ All this amounts to the following result.
\begin{thm}\label{T:1-D}
If $T$ is a contraction acting on $\cH$, $\Pi_D$ is the isometry as in \eqref{1-D-PiD}, and $V_D$ is the isometry defined on 
$$
\mathcal{K_D}:=
\begin{bmatrix}
 H^2(\mathcal{D}_{T^*})\\ \cQ_{T^*}
\end{bmatrix}
$$ by 
$$V_D
=\begin{bmatrix}
M_z&0\\
0& W_D
\end{bmatrix}:
\begin{bmatrix}
 H^2(\mathcal{D}_{T^*})\\ \cQ_{T^*}
\end{bmatrix}\to
\begin{bmatrix}
 H^2(\mathcal{D}_{T^*})\\ \cQ_{T^*}
\end{bmatrix},
$$then
$V_D$ is a minimal isometric lift of $T$ via the embedding $\Pi_D$.
\end{thm} The minimality of the lift, i.e., the fact that
$$
\cK_D=\begin{bmatrix}
 H^2(\mathcal{D}_{T^*})\\ \cQ_{T^*}
\end{bmatrix}
=\bigvee_{n\geq0}\begin{bmatrix}
M_z^n&0\\
0&W_D^n
\end{bmatrix}\begin{bmatrix}
\cO_{D_{T^*},T^*}\\
Q_{T^*}
\end{bmatrix}\cH
$$ is a result of \eqref{QT-star} and the following identity 
\begin{align}\label{DougMin}
\cO_{D_{T^*},T^*} - M_z \cO_{D_{T^*},T^*}T^* = D_{T^*}.
\end{align} For more details, the readers are referred to the original paper \cite{Doug-Dilation} by Douglas and the recent monograph \cite[Chapter 2]{BS-CUP}.

The goal of this section is to find a functional model for isometric lifts of a $q$-commuting contractive pair in this spirit of the Douglas' form as in Theorem \ref{T:1-D}. We need a pair $(W_1,W_2)$ of $q$-commuting unitary operators canonically associated to a $q$-commuting contractive pair $(T_1,T_2)$, which would give us $W_1W_2=qW_2W_1= W_D$, where $W_D$ is obtained as in the analysis above for the contraction operator $T=T_1T_2$.
\subsection{Canonical $q$-commuting unitary pair}\label{SS:CanonUniPair}
Suppose that $(T_1,T_2)$ is a $q$-commuting contractive pair acting on $\cH$ and $T=T_1T_2$. Let the positive semi-definite contraction $Q_{T^*}$ be as in \eqref{Q}. We note by using part (2) of Lemma \ref{lem:Prod} that 
	$$\langle T_iQ_{T^*}^2T^*_i h,h\rangle=\langle\lim T_iT^nT^{*n}T_i^* h,h\rangle = q^n\oq^n\langle \lim  T^n T_iT^*_iT^{*n}h,h \rangle\le \langle Q_{T^*}^2h,h \rangle $$  for all $h\in \mathcal{H}$ and $i=1,2$. This gives way to the existence of contractions $X^*_i :\overline{\operatorname{Ran}Q_{T^*}}\to \overline{\operatorname{Ran}Q_{T^*}}$ defined densely by 
	\begin{align}\label{X} 
		X_i^*Q_{T^*}h=Q_{T^*}T_i^*h \text{ for every }h\in\cH.
	\end{align} The contractive operators $X_1,X_2$ have the following properties: for every $h\in \cH$,
	\begin{align*} 
		&X_1^*X_2^*Q_{T^*}h=X_1^*Q_{T^*}T_2^*h=Q_{T^*}T_1^*T_2^*h=q Q_{T^*}T^*h=qX^*Q_{T^*}h \text{ and}\\ 
		&X_2^*X_1^*Q_{T^*}h=X_2^*Q_{T^*}T_1^*h=Q_{T^*}T_2^*T_1^*h=Q_{T^*}T^*=X^*Q_{T^*}h.
	\end{align*}
	
	Thus $(X_1^*,X_2^*)$ a $q$-commuting pair of contractions with their product $qX^*$ being an isometry on $\overline{\operatorname{Ran}Q_{T^*}}$. Consequently by part (1) of Lemma \ref{lem:Prod}, $(X_1^*,X_2^*)$ is a $q$-commuting pair of isometries. Applying Theorem \ref{isoext}, we obtain a $q$-commuting pair $(W_{1}^*,W_{2}^*)$ of unitaries such that $W_{1}^*W_{2}^*=qW_D^*$. Here we use the fact that $qW_D^*$ is the minimal unitary extension of the isometry $qX^*$.

	\begin{definition} \label{D:CanonUni}
		For a $q$-commuting pair of contractions \((T_1, T_2)\), the corresponding $q$-commuting pair of unitaries \((W_1, W_2)\) obtained as described above will be referred to as the {\em canonical $q$-commuting pair of unitaries} associated with \((T_1, T_2)\). 
	\end{definition}
	The canonicity is justified in the following theorem.
	\begin{thm}\label{T:cano}
		Let $(T_1,T_2)$ on $\cH$ and $(T_1',T_2')$ on $\mathcal{H'}$ be two \(q\)-commuting pairs of contractions with their respective canonical $q$-commuting pair being $(W_{1},W_{2})$ on $\cQ_{T^*}$ and $(W_{1}',W_{2}')$ on $\cQ_{T'^*}$. If $(T_1,T_2)$ and $(T_1',T_2')$ are unitarily equivalent, then so are $(W_{1},W_{2})$ and $(W_{1}',W_{2}')$. In particular, if $(T_1, T_2) = (T_1', T_2')$, then 
		$(W_{1},W_{2}) = (W_{1}',W_{2}')$. 
	\end{thm}
	\begin{proof}
		
		Let \(\psi: \mathcal{H} \to \mathcal{H'}\) be a unitary operator that intertwines \((T_1, T_2)\) and \((T_1', T_2')\). As before, define \(T = T_1 T_2\) and \(T' = T_1' T_2'\). Let \(Q_{T^*}\) and \(Q_{{T'}^*}\) be the positive semi-definite contractions, as described in \eqref{Q}, associated with \((T_1, T_2)\) and \((T_1', T_2')\), respectively. It is clear that \(\psi\) intertwines \(Q_{T^*}\) and \(Q_{{T'}^*}\), and therefore \(\psi\) maps \(\overline{\operatorname{Ran} Q_{T^*}}\) onto \(\overline{\operatorname{Ran} Q_{{T'}^*}}\). We denote the restriction of \(\psi\) to \(\overline{\operatorname{Ran} Q_{T^*}}\) by \(\psi\) itself. Let \((X_1, X_2)\) act on \(\overline{\operatorname{Ran} Q_{T^*}}\) and \((X_1', X_2')\) act on \(\overline{\operatorname{Ran} Q_{{T'}^*}}\) be the \(q\)-commuting pairs of co-isometries corresponding to the pairs \((T_1, T_2)\) and \((T_1', T_2')\), as mentioned at the beginning of this section. It is straightforward to verify from the definition that 
		$$ 
		\psi(X_1, X_2) = (X_1', X_2') \psi.
		$$ 
		Recall that \( W_D^* \) on \(\cQ_{T^*} \) is the minimal unitary extension of the isometry \( X^* \) as defined in \eqref{X}. Therefore, 
		$$ 
		\cQ_{T^*} = \overline{\text{span}}\{W_D^n x : x \in \overline{\operatorname{Ran} Q_{T^*}} \text{ and } n \geq 0\}.
		$$
		
		Given the analogous representation of \(\cQ_{T'^*}\), we define \(\tau_\psi:\cQ_{T^*} \to\cQ_{T'^*}\) by 
		\begin{align}\label{taupsi}  
			\tau_\psi: W_D^n x \mapsto W_D'^n \psi x, \quad \text{for all } x \in \overline{\operatorname{Ran} Q_{T^*}}\text{ and } n \geq 0. 
		\end{align} 
		This map is then extended linearly and continuously. It is clear that \(\tau_\psi\) is a unitary operator that intertwines \(W_D\) and \(W_D'\).

		For \( x \in \overline{\operatorname{Ran} Q_{T^*}} \) and \( n \geq 0 \), we have the following:
		
		\begin{align*} 
			\tau_\psi W_1(W_D^n x)= q \tau_\psi W_D^n (W_1 x)
			= q \tau_\psi W_D^{n+1}(W_2^* x) 
			= q W_D'^{n+1} \psi(X_2^* x) 
			&=q W_D'^{n+1} W_2'^* \psi x \\
			&= W_1' W_D'^n \psi x \\
			&= W_1' \tau_\psi(W_D^n x).
		\end{align*}
		
		A similar calculation shows that \(\tau_\psi\) also intertwines \(W_2\) and \(W_2'\).
		Note that if, in the above theorem, we assume \((T_1, T_2) = (T_1', T_2')\), then we have \(\cQ_{T^*} =\cQ_{T'^*}\), \(\psi = I_{\mathcal{H}}\), and \(\tau_\psi = I_{\cQ_{T^*}}\).   
	\end{proof}
	%%%%%%%%%%%%%%%%%%%%%%%%%%%%%%%%%%%%%%%%
	
	The following is a direct consequence of Theorem \ref{T:cano}.
	
	\begin{corollary} \label{uniquecano}
		For a \(q\)-commuting pair of contractions \((T_1, T_2)\), let \(Q_{T^*}\) be the positive semi-definite contraction and \(W_D\) be the unitary, as discussed earlier. Suppose \((W'_1, W'_2)\) is a pair of \(q\)-commuting unitaries such that for each \(j = 1, 2\),
		$${W'_j}^* Q_{T^*} = Q_{T^*}T_j^* \quad \text{and} \quad {W'_1}^* {W'_2}^* =q{W}_D^*.$$
		Then \((W'_1, W'_2) = (W_1, W_2)\).  
	\end{corollary}

	%%%%%%%%%%%%%%%%%%%%%%%%%%%%%%%%%%%%%%%%
\subsection{A Douglas-type model for isometric lifts}
	We now have all the necessary tools to derive the main result of this section. 
	\begin{thm} \label{T:Dmod}
		Let $(T_1,T_2)$ be a pair of $q$-commuting contractions on $\cH$. Let $(W_{1},W_{2})$ be the canonical $q$-commuting pair of unitary operators for $(T_1,T_2)$ and $(\cF_*,\Lambda_*,P_*,U_*)$ be a special And\^o tuple for $(T_1^*,T_2^*)$ (see Definition \ref{D:SpecialAndo}). Consider the pair of operators 
		\begin{align}\label{Dmod} 
			(V^D_1,V^D_2)=\left(
			\begin{bmatrix}
			M_{U_*^*(P_*^\perp + z P_*)}R_q&0\\ 0&W_{1}
			\end{bmatrix},
			\begin{bmatrix}
			R_{\oq}M_{(P_* + z P_*^\perp)U_*}&0\\
			0&W_{2}
			\end{bmatrix}\right) \quad\mbox{on}\quad \begin{bmatrix}H^2(\cF_*)\\ \cQ_{T^*}\end{bmatrix}.
		\end{align}Let $\Pi^D:\cH\to \sbm{H^2(\cF_*)\\ \cQ_{T^*}}$ be defined by 
		\begin{align}\label{GenPi} 
			\Pi^D h:=\begin{bmatrix}
			I_{H^2}\otimes \Lambda_* & 0\\
			0 & I_{\cQ_{T^*}}
			\end{bmatrix}\Pi_Dh= \begin{bmatrix}
			\sum_{n=0}^{\infty}z^n\Lambda_* D_{T^*}T^{* n}h\\
			 Q_{T^*}h
			\end{bmatrix}. 
		\end{align} Then $\Pi^D$ is an isometry and $(\Pi^D;V^D_1,V^D_2)$ is a $q$-commuting isometric lift of $(T_1,T_2)$.
		
Conversely, if $(\bV_1,\bV_2)$ acting on $\bcK$ is any minimal isometric lift of $(T_1,T_2)$ via an isometric embedding $\bPi:\cH\to\bcK$, then there is an And\^o tuple $(\cF_*,P_*,\Lambda_*,U_*)$ so that $(\bV_1,\bV_2)$ is unitarily equivalent to a pair of the form as in \eqref{Dmod}, and the isometric embedding $\bPi$ must be of the form as in \eqref{GenPi}.
	\end{thm} 
	\begin{proof}
		The operator \(\Pi^D\) as in \eqref{GenPi} is an isometry since it is the product of two isometries. It is clear from Theorem \ref{T:qBCL} that \((V^D_1, V^D_2)\) forms a pair of \(q\)-commuting isometries on \(H^2(\mathcal{F}_*) \oplus\cQ_{T^*}\). Therefore, it suffices to show that the above pair lifts \((T_1, T_2)\), meaning that \({V_j^D}^* \Pi^D = \Pi^D T_j^*\) for all \(j = 1, 2\). For all $h\in \mathcal{H}$,
		
 \begin{align*}
    V^{D*}_1\Pi^D h &=\begin{bmatrix}M_{U^*_*P^\perp_*+zU^*_*P_*} R_q& 0\\0 &W_1\end{bmatrix}^*\Pi^D h 
    \\&=\begin{bmatrix}R_{\bar{q}} \otimes P^\perp_*U_*+R_{\bar{q}} M_z^*\otimes P_*U_* &0\\0& W_1^*\end{bmatrix}\begin{bmatrix}\sum_{n\ge 0} z^n\Lambda_*D_{T^*}{T^*}^n h \\ Q_{T^*}h\end{bmatrix}
    \\&=\begin{bmatrix}
    \sum_{n\ge 0} \bar{q}^n z^n P^\perp_*U_*\Lambda_*D_{T^*}{T^*}^n h+ \sum_{n\ge 1}\bar{q}^n z^{n-1}P_*U_*\Lambda_*D_{T^*}{T^*}^n h\\ W^*_1 Q_{T^*}h \end{bmatrix}
   \\&=\begin{bmatrix}\sum_{n\ge 0}\bar{q}^n z^n(P^\perp_*U_*\Lambda_*D_{T^*}+P_*U_*\Lambda_*D_{T^*}T^*){T^*}^n h\\ W^*_1 Q_{T^*}h\end{bmatrix}.
\end{align*}

 It follows from \eqref{Prop2}, \eqref{X} and Lemma \ref{lem:Prod} that the above expression is same as 
  \begin{align*}
     \begin{bmatrix}\sum_{n\ge 0}\bar{q}^n z^n \Lambda_*D_{T^*}T^*_1{T^*}^n h\\ Q_{T^*}T^*_1 h\end{bmatrix} 
    = \begin{bmatrix}\sum_{n\ge 0} \bar{q}^n{q}^nz^n \Lambda_*D_{T^*}{T^*}^n T^*_1h\\ Q_{T^*}T^*_1 h  \end{bmatrix} =\Pi^D T^*_1 h.\end{align*}
		Proceeding in a similar manner by combining \eqref{Prop2}, \eqref{X}, and Lemma \ref{lem:Prod} gives the intertwining $V_2^{D*}\Pi^D=\Pi^D T_2^*$.
		
We now prove the converse direction. The first step is to apply part (3) of Theorem \ref{T:qBCL}: there exist a projection-unitary pair $(P_*,U_*)$ in $\cB(\cD_{V^*})$, and a $q$-commuting unitary pair $(W_1,W_2)$ acting on $\cap_{n\geq0} V^n\cK$ so that the original isometric lift $(\bV_1,\bV_2)$ is jointly unitarily equivalent to
\begin{align*}
(V_1,V_2)=\left(\begin{bmatrix}
 M_{ U_*^*P_*^\perp +z U_*^*P_*}R_q & 0\\ 0 & W_{1} \end{bmatrix}, 
\begin{bmatrix}
R_\oq M_{P_* U_* + z P_*^\perp U_*}&0\\0&W_{2}
\end{bmatrix}\right)\quad\mbox{ on }
\begin{bmatrix}
H^2(\cD_{\bV^*})\\
\cQ_{\bV^*}
\end{bmatrix}
\end{align*}via the unitary identification map
$$
\tau_{\rm W}:k\mapsto \begin{bmatrix}
D_{\bV^*}(I_\cK-z\bV^*)^{-1}\\
Q_{\bV^*}
\end{bmatrix}k.
$$
Since we are interested in finding a unitary copy of the original isometric lift $(\bV_1,\bV_2)$, we assume without loss of generality that the isometric lift is $(V_1,V_2)$ as in the $2\times2$ block matrix form above and the isometric embedding $\Pi$ of $\cH$ is now given by 
\begin{align}\label{NewPi}
\Pi=\tau_{\rm W} \bPi = \begin{bmatrix}
D_{V^*}(I_\cK-zV^*)^{-1}\\
Q_{V^*}
\end{bmatrix}\bPi.
\end{align}
Let us note that with $W=W_1W_2$,
\begin{align}\label{K1}
\cK_{1}=\bigvee_{n\geq0}\begin{bmatrix}
M_z^{\cD_{\bV^*}}&0\\
0&W
\end{bmatrix}^n\Pi\cH\subset \begin{bmatrix}
H^2(\cD_{\bV^*})\\
\cQ_{\bV^*}
\end{bmatrix}
\end{align}is a minimal isometric lift of $T=T_1T_2$. Since any two minimal isometric lifts of a single contraction operator are unitarily equivalent (see e.g., \cite{Nagy-Foias}, Chapter I, Theorem \(4.1\)), there is a unitary operator
$$
\tau:\begin{bmatrix}
H^2(\cD_{T^*})\\
\cQ_{T^*}
\end{bmatrix}\to \cK_1
$$such that
\begin{align}\label{DougKey}
&\tau
\begin{bmatrix}
M_z^{\cD_{T^*}} & 0\\
0& W_D
\end{bmatrix}=
\begin{bmatrix}
M_z^{\cD_{\bV^*}}&0\\
0& W
\end{bmatrix}\tau \quad\mbox{and}\quad\\
&\tau\Pi_D=\tau\begin{bmatrix}
D_{T^*}(I-zT^*)^{-1}\\
Q_{T^*}
\end{bmatrix}=\begin{bmatrix}
D_{\bV^*}(I-z\bV^*)^{-1}\\
Q_{\bV^*}
\end{bmatrix}\bPi=\Pi.\label{DougKey1}
\end{align}
Let us view $\tau$ as an isometry from $\sbm{H^2(\cD_{T^*})\\ \cQ_{T^*}}\to \sbm{H^2(\cD_{\bV^*})\\ \cQ_{\bV^*}}$ with $\operatorname{Ran}\tau = \cK_1$. Let us write
$$
\tau= \begin{bmatrix}
A&B\\
C&D
\end{bmatrix}:\begin{bmatrix}
H^2(\cD_{T^*})\\
\cQ_{T^*}
\end{bmatrix} \to
\begin{bmatrix}
H^2(\cD_{\bV^*})\\
\cQ_{\bV^*}
\end{bmatrix}.
$$Then from equation \eqref{DougKey} we have
\begin{align}\label{4Ints}
AM_z^{\cD_{T^*}}=M_z^{\cD_{\bV^*}}A,\quad BD_D=M_z^{\cD_{\bV^*}}B,\quad CM_z^{\cD_{T^*}}=WC,\quad\mbox{and}\quad DW_D=WD.
\end{align}By Lemma \ref{L:Zero}, $B=0$. This together with \eqref{DougKey1} gives
\begin{align}\label{The2ndOne}
&AD_{T^*}(I-zT^*)^{-1} = D_{\bV^*}(I-z\bV^*)^{-1}\bPi \quad\mbox{and}\quad\notag\\
&CD_{T^*}(I-zT^*)^{-1}+DQ_{T^*}=Q_{\bV^*}\bPi.
\end{align}The first of the two equations above, the first intertwining in \eqref{4Ints} and the identity \eqref{DougMin} gives
\begin{align*}
AD_{T^*}=D_{\bV^*}\bPi.
\end{align*}This shows that $A$ takes the space of constant functions of $H^2(\cD_{T^*})$ into that of $H^2(\cD_{\bV^*})$. This and the intertwining $AM_z^{\cD_{T^*}}=M_z^{\cD_{\bV^*}}A$ together forces $A$ to be of the form
$$
A=I_{H^2}\otimes\Lambda_*
$$for some isometry $\Lambda_*:\cD_{T^*}\to\cD_{\bV^*}$. Thus the isometry $\tau$ has the property that 
$$
A=P_{H^2(\cD_{\bV^*})}\tau|_{H^2(\cD_{T^*})}
$$ is also an isometry. This forces 
$$
C=P_{\cQ_{\bV^*}}\tau|_{H^2(\cD_{T^*})} = 0.
$$ Consequently, we have
$$
\tau=\begin{bmatrix}
I_{H^2}\otimes \Lambda_* & 0\\
0& D
\end{bmatrix}
$$for some isometries $\Lambda_*:\cD_{T^*}\to\cD_{\bV^*}$ and $\omega:\cQ_{T^*}\to\cQ_{\bV^*}$.

Since $\operatorname{Ran}{\tau}=\cK_1$, from the expression \eqref{K1} of the space $\cK_1$ and \eqref{NewPi}, we have
\begin{align*}
\operatorname{Ran}D=P_{\cQ_{\bV^*}}\cK_1= \bigvee_{n\geq0}W^n\Pi\cH=\bigvee_{n\geq0}W^nQ_{\bV^*}\bPi\cH.
\end{align*}The last intertwining in \eqref{4Ints} is same as $W^*D=DW_D^*$, and therefore $\operatorname{Ran}D$, which contains $P_{\cQ_{\bV^*}}\Pi\cH$ as a subspace, must be reducing for $W$. We started with a minimal isometric lift $(\bV_1,\bV_2)$ and so we must have
$$
\operatorname{Ran}D=\cQ_{\bV^*},
$$in other words, $D:\cQ_{T^*}\to\cQ_{\bV^*}$ must be a unitary.

%	\begin{align*}
%			V^{D*}_2\Pi h &=({R_{\oq}M_{P_*U_*+zP^\perp_*U_*}   \oplus W_2)}^*\Pi h
%			\\&= ((R_q\otimes U^*_*P_*+M_z^*R_q\otimes U^*_*P^\perp_*)\oplus W_2^*)\Big(\sum_{n\ge 0} z^n\Lambda_*D_{T^*}{T^*}^n h \oplus Q_{T^*}h\Big)
%			\\&=(
%			\sum_{n\ge 0} q^n z^n U^*_*P_*\Lambda_*D_{T^*}{T^*}^n h+ \sum_{n\ge 1}q^n z^{n-1}H^*_1\Lambda_*D_{T^*}{T^*}^n h)\oplus W^*_2 Q_{T^*}h
%			\\&=(\sum_{n\ge 0} q^n z^n(U^*_*P_*\Lambda_*D_{T^*}+q U^*_*P^\perp_*\Lambda_*D_{T^*}T^*){T^*}^n h)\oplus X^*_2 Q_{T^*}h\\&=\sum_{n\ge 0}q^n z^n \Lambda_*D_{T^*}T^*_2{T^*}^n h\oplus Q_{T^*}T^*_2 h
%			\\&=\sum_{n\ge 0} z^n \Lambda_*D_{T^*}{T^*}^n T^*_2h\oplus Q_{T^*}T^*_2 h  
%			\\&=\Pi T^*_2 h\end{align*}
%		The two equalities above lead to the conclusion that \((\Pi; V^D_1, V^D_2)\) is a \(q\)-commuting isometric dilation of \((T_1, T_2)\).
	\end{proof}
	%%%%%%%%%%%%%%%%%%%%%%%%%%%%%%%%%%%%%%%%%%%%%%%%%%%%%%%%%%%%%%%%%%
	\begin{remark}
		
		The operators \( U_1 \) and \( U_2 \) on \( L^2(\cF_*) \oplus \cQ_{T^*} \) defined as follows:
		
		\[
		U_1 := \begin{bmatrix} M_{U^*_*P^\perp_*+\zeta U^*_*P_*} R_q & 0 \\ 0 & W_1 \end{bmatrix}, \quad U_2 := \begin{bmatrix} R_{\oq} M_{P_*U_*+\zeta P^\perp_*U_*} & 0 \\ 0 & W_2 \end{bmatrix}
		\]
		are \( q \)-commuting unitary extensions of \( V_1^D \) and \( V_2^D \), respectively. Therefore, the pair \( (U_1, U_2) \) represents a \( q \)-commuting unitary dilation of \( (T_1, T_2) \).

	\end{remark}
We conclude this section with a corollary to the above theorem. This will be utilized in the subsequent section for the development of the Sz.-Nagy--Foias-type functional model for $q$-commuting contractive pairs. This corollary follows from the intertwinings $(V_1^{D*},V_2^{D*})\Pi^D=\Pi^D(T_1^*,T_2^*)$ established in the proof of the forward direction of Theorem \ref{T:Dmod} and the matrix representations \eqref{Dmod} and \eqref{GenPi} of the operators involved in the intertwinings. 
	\begin{corollary}\label{intphiD}
		Given a pair \( (T_1, T_2) \) of \( q \)-commuting contractions, the isometry \( \Pi_D \) satisfies the following intertwining relations:
		\begin{align*}
			&\Pi_D T_1^* = \begin{bmatrix}
			R_{\oq} M^*_{\Lambda_*^* U_*^* P_*^\perp \Lambda_* + z \Lambda_*^* U_*^* P_* \Lambda_*} &0\\
			0& W_1^*
\end{bmatrix}\Pi_D = \begin{bmatrix}
			R_{\oq} M^*_{G_1^* + z G_2} &0\\
			0& W_1^*
\end{bmatrix}\Pi_D\quad\mbox{and}\\
&
\Pi_D T_2^* = \begin{bmatrix}
 M^*_{\Lambda_*^* P_* U_* \Lambda_* + z \Lambda_*^* P_*^\perp U_* \Lambda_*} R_q &0\\ 0&W_2^* \end{bmatrix} \Pi_D=
 \begin{bmatrix}
 M^*_{G_2^* + z G_1} R_q &0\\ 0&W_2^* 
 \end{bmatrix}\Pi_D,
		\end{align*}
		where we used the short-hand notation
		\begin{align}\label{Gs}
(G_1,G_2)=\Lambda_*^* (P_*^\perp U_*, U_*^*P_* )\Lambda_*.
		\end{align}
	\end{corollary}
%	
%	%%%%%%%%%%%%%%%%%%%%%%%%%%%%%%%%%%%%%%%%%%%%%%%%%%
%	
%	\begin{proof}
%		
%		Using the Douglas-type model of \( (T_1, T_2) \) described in the Theorem \ref{Dmod} and the isometry \(\Pi = \left( (I_{H^2} \otimes \Lambda_*) \oplus I_{\cQ_{T^*}} \right) \Pi_D\), we derive the following equality:
%		
%		
%		
%		
%		\begin{align*}\Pi_D T^*_1&=\left( (I_{H^2} \otimes \Lambda^*_*) \oplus I_{\cQ_{T^*}} \right) \Pi T^*_1\\
%			&=\left( (I_{H^2} \otimes \Lambda^*_*) \oplus I_{\cQ_{T^*}} \right){(M_{U^*_*P^\perp_*+zU^*_*P_*} R_q   \oplus W_1)}^*\Pi \\
%			&=\left( (I_{H^2} \otimes \Lambda^*_*) \oplus I_{\cQ_{T^*}} \right){(R_{\oq} M^*_{U^*_*P^\perp_*+zU^*_*P_*}  \oplus W^*_1)}\left( (I_{H^2} \otimes \Lambda_*) \oplus I_{\cQ_{T^*}} \right) \Pi_D\\
%			&=\left( R_{\oq} M^*_{\Lambda_*^* U_*^* P_*^\perp \Lambda_* + z \Lambda_*^* U_*^* P_* \Lambda_*} \oplus W_1^* \right) \Pi_D
%		\end{align*}
%		
%		The second intertwining relation follows through an analogous argument.
%		
%	\end{proof}

The operator pair $(G_1,G_2)$ as in \eqref{Gs} holds deeper significance for a $q$-commuting contractive pair $(T_1,T_2)$. It seems that the pair depends on the choice of the And\^o tuple $(\cF_*,\Lambda_*,P_*,U_*)$ but as we shall see in the next section, they are uniquely determined by the $q$-commuting contractive pair $(T_1,T_2)$.

\subsection{A Sz.-Nagy--Foias-type model for isometric lifts}

Here we develop yet another type of functional model for the minimal isometric lifts of a \(q\)-commuting contractive pair but in the spirit of the functional model developed by Sz.-Nagy and Foias \cite{Nagy-Foias} for single contractions. We begin by reviewing the work of Sz.-Nagy and Foias.
	
	Given a contraction \( T \) on a Hilbert space \( \mathcal{H} \), consider the map  \( \Theta_T: \mathbb{D} \to \mathcal{B}(\mathcal{D}_T, \mathcal{D}_{T^*}) \) defined by
	\begin{align}
		\Theta_T(z) = \left[ -T + z D_{T^*} \left( I_{\mathcal{H}} - z T^* \right)^{-1}D_T \right]|_{\cD_T}:\cD_T\to\cD_{T^*} \quad \text{for }z\in \mathbb{D}.
	\end{align}
	
The well-definedness and the analyticity of the map \(\Theta_T\) is a consequence of the convergence of the Neumann series and the relation between the defect operators
$$
TD_T = D_{T^*}T.
$$ As demonstrated in \cite[Chapter VI]{Nagy-Foias}, it turns out that $\Theta_T(z)$ is a contraction for each $z$ in $\mathbb D$ and that it is purely contractive, i.e.,
$$
\|\Theta_T(0)f\| < \|f\| \quad\mbox{ for all}\quad 0\neq f \in \cD_T.
$$ 
The $\cB(\cD_T,\cD_{T^*})$-valued analytic function $\Theta_T$ is referred to as the \textit{characteristic function} for \(T\). The characteristic functions is a unitary invariant for \textit{cnu} contractions in the sense that \textit{two cnu contractions $T,T'$ are unitarily equivalent to if and only if their characteristic functions $\Theta_T$ and $\Theta_{T'}$ coincide, i.e., there are unitary operators $u:\cD_{T}\to\cD_{T^*}$ and $u_*:\cD_{T'}\to\cD_{T'^*}$ such that 
\begin{align}\label{CharcCoin}
u_{*}\Theta_T(z)=\Theta_{T'}(z)u \quad\mbox{for all}\quad z\in\bD.
\end{align}}
Given a contractive operator $T$, we associate the $\cB(\cD_T)$-valued function \(\Delta_{\Theta_T}\), which is defined almost everywhere on the unit circle \(\mathbb{\mathbb T}\) by \begin{align}\Delta_{\Theta_T}(\zeta) := (I - \Theta_T(\zeta)^* \Theta_T(\zeta))^{1/2}.\end{align}
Here $\Theta_T(\zeta)$ is defined as the radial limit
$$
\Theta_T(\zeta)=\lim_{r\to1-}\Theta(r\zeta)\quad\mbox{for }\zeta\in\bT
$$ of the characteristic function; it is known (see \cite[Chapter V]{Nagy-Foias}) that the radial limit exists almost everywhere in $\mathbb T$ for any bounded operator-valued analytic function. 

It is possible to find a functional model for isometric lifts of a cnu contraction $T$ involving the characteristic function for $T$.  Let $\cQ_{T^*}$ be the Hilbert space as appeared in \eqref{QT-star} for the analysis of Douglas model for isometric lifts. It is known (see for example \cite[Chapter 2]{BS-CUP}) that there is a unitary $\omega_{\rm{D,NF}}:\cQ_{T^*}\to\overline{\Delta_{\Theta_T } L^2(\cD_T)}$ so that the following intertwining holds:
\begin{align}\label{omega-DNF}
\omega_{\rm{D,NF}} \cdot W_D = M_\zeta|_{\overline{\Delta_{\Theta_T } L^2(\cD_T)}}\cdot \omega_{\rm{D,NF}}.
\end{align}Consider the isometric embedding
$$
\Pi_{\rm NF}:\cH\to \begin{bmatrix}
H^2(\cD_{T^*})\\
\overline{\Delta_{\Theta_T } L^2(\cD_T)}
\end{bmatrix}
$$defined by 
\begin{align}\label{NF-IsoEmbed}
\Pi_{\rm NF}:=\begin{bmatrix}
I_{H^2(\cD_{T^*})}&0\\
0& \omega_{\rm{D,NF}}
\end{bmatrix}\Pi_D=
\begin{bmatrix}
D_{T^*}(I_\cH-zT^*)^{-1}\\
\omega_{\rm{D,NF}}Q_{T^*}
\end{bmatrix}
\end{align}and the isometric operator
\begin{align}\label{NF-Iso}
V_{\rm NF}:=
\begin{bmatrix}
M_z&0\\
0& M_\zeta|_{\overline{\Delta_{\Theta_T } L^2(\cD_T)}}
\end{bmatrix}:
\begin{bmatrix}
H^2(\cD_{T^*})\\
\overline{\Delta_{\Theta_T } L^2(\cD_T)}
\end{bmatrix}\to
\begin{bmatrix}
H^2(\cD_{T^*})\\
\overline{\Delta_{\Theta_T } L^2(\cD_T)}
\end{bmatrix}.
\end{align}
The following functional model for isometric lifts of a single contractive operator is due to Sz.-Nagy and Foias. We arrive at this by the discussion above and the fact that $(\Pi_D,\sbm{M_z&0\\ 0& W_D})$ is a minimal isometric lift acting on the space $\cK_D=\sbm{H^2(\cD_{T^*})\\ \cQ_{T^*}}$.
	\begin{thm}[Sz.-Nagy and Foias]
		Let $T$ be a cnu contraction on a Hilbert space $\cH$. Then the isometric operator $V_{\rm NF}$ as in \eqref{NF-Iso} is a minimal isometric lift of $T$ acting on
		\begin{align}\label{KthetaT}
		\cK_{\Theta_T}:= \begin{bmatrix}
H^2(\cD_{T^*})\\
\overline{\Delta_{\Theta_T } L^2(\cD_T)}
\end{bmatrix}
		\end{align}via the embedding $\Pi_{\rm NF}$ as in \eqref{NF-IsoEmbed}, i.e.,
		\begin{align}\label{NF-Int}
V_{\rm NF}^*\Pi_{\rm NF}=\Pi_{\rm NF} T^*.
		\end{align}
%		 ============ is a minimal isometric dilation of $T$, where $V_{\text{NF}}:= M_z\oplus M_\zeta|_{\overline{\Delta_{\Theta_T } L^2(\cD_T)}} $ is an isometry on $\cK_{\Theta_T}:= H^2(\cD_{T^*})\oplus \overline{\Delta_{\Theta_T} L^2(\cD_T)}$ and $\Pi_{\text{NF}}:\cH\to \cK_{\Theta_T}$ is also an isometry with $\cH_{\Theta_T}:= \operatorname{ran}\Pi_{\text{NF}}=(H^2(\cD_{T^*})\oplus {\overline{\Delta_{\Theta_T} L^2(\cD_T)}})\ominus(\Theta_T H^2(\cD_T)\oplus \Delta_{\Theta_T} H^2(\cD_T)).$
	\end{thm}
With $(W_1,W_2)$ as the canonical $q$-commuting unitary pair as in Definition \ref{D:CanonUni} and the unitary operator $\omega_{\rm{D,NF}}$ as in \eqref{omega-DNF}, let us define
	\begin{align}\label{unifunda}
		W^{\textrm{NF}}_i := \omega_{\rm{D,NF}} W_i \omega_{\rm{D,NF}}^* \quad \text{for} \quad i = 1, 2.
	\end{align}
	Then we have the following.
	\begin{thm} \label{T:NFmod}
		Let $(\!T_1,T_2\!)$ be a pair of $q$-commuting contractions on $\cH$ with the product operator \(T=T_1T_2\) being a cnu contraction. Let $(W^{\textrm{NF}}_{1},W^{\textrm{NF}}_{2})$ be as defined in \eqref{unifunda} and $(\cF_*,\Lambda_*,P_*,U_*)$ be a special And\^o tuple for $(T_1^*,T_2^*)$. Consider the pair of operators 
		\begin{align}\label{NFmod} 
			(V^{\rm{NF}}_1,V^{\rm{NF}}_2)=\left(
			\begin{bmatrix}
			M_{U_*^*(P_*^\perp + z P_*)}R_q&0\\ 0&W^{\textrm{NF}}_{1}
			\end{bmatrix},
			\begin{bmatrix}
			R_{\oq}M_{(P_* + z P_*^\perp)U_*}&0\\
			0&W^{\textrm{NF}}_{2}
			\end{bmatrix}\right)\end{align}
 on $ \begin{bmatrix}H^2(\cF_*)\\ \overline{\Delta_{\Theta_T } L^2(\cD_T)}\end{bmatrix} $ and let $\Pi^{\textrm{NF}}:\cH\to \sbm{H^2(\cF_*)\\ \overline{\Delta_{\Theta_T } L^2(\cD_T)}}$ be defined by 
		\begin{align}\label{NFGenPi} 
			\Pi^{\textrm{NF}} h:=\begin{bmatrix}
			I_{H^2}\otimes \Lambda_* & 0\\
			0 & I_{\overline{\Delta_{\Theta_T } L^2(\cD_T)}}
			\end{bmatrix}\Pi_{\rm{NF}}h= \begin{bmatrix}
			\sum_{n=0}^{\infty}z^n\Lambda_* D_{T^*}T^{* n}h\\
			 \omega_{\rm{D,NF}}Q_{T^*}h
			\end{bmatrix}. 
		\end{align} Then $\Pi^{\textrm{NF}}$ is an isometry and $(\Pi^{\textrm{NF}}; V^{\rm{NF}}_1, V^{\rm{NF}}_2)$ is a $q$-commuting isometric lift of $(T_1,T_2)$.
		
Conversely, if $(\bV_1,\bV_2)$ acting on $\bcK$ is any minimal isometric lift of $(T_1,T_2)$ via an isometric embedding $\bPi:\cH\to\bcK$, then there is an And\^o tuple $(\cF_*,P_*,\Lambda_*,U_*)$ so that $(\bV_1,\bV_2)$ is unitarily equivalent to a pair of the form as in \eqref{NFmod}, and the isometric embedding $\bPi$ must be of the form as in \eqref{NFGenPi}.
	\end{thm} 

	\begin{proof} Let us consider the unitary operator \(\tau:\sbm{H^2(\cF_*)\\ \cQ_{T^*}} \to \sbm{H^2(\cF_*)\\ \overline{\Delta_{\Theta_T } L^2(\cD_T)}}  \) defined by 
$$\tau h=\begin{bmatrix} I &0\\0& \omega_{\rm{D,NF}} \end{bmatrix}h.$$ Then it follows that $\Pi^{\textrm{NF}}=\tau \Pi^D $ 
 and $V^{\rm{NF}}_i =\tau V^D_i \tau^*$ for all \(i=1,2\). From the forward direction of the Theorem \ref{T:Dmod}, we conclude that $(\Pi^{\textrm{NF}}; V^{\rm{NF}}_1, V^{\rm{NF}}_2)$ is a $q$-commuting isometric lift of $(T_1,T_2)$. Thus  $(\Pi^{\textrm{NF}}; V^{\rm{NF}}_1, V^{\rm{NF}}_2)$ and  $(\Pi^D; V^D_1, V^D_2)$ are two unitarily equivalent \(q\)-commuting isometric lifts of $(T_1,T_2).$

The converse follows from the equivalance of the isometric lifts  $(\Pi^{\textrm{NF}}; V^{\rm{NF}}_1, V^{\rm{NF}}_2)$ and  $(\Pi^D; V^D_1, V^D_2)$ of the \(q\)-commutative contractive pair \((T_1,T_2)\) and the converse of the Theorem \ref{T:Dmod}.
	\end{proof}

	%%%%%%%%%%%%%%%%%%%%%%%%%%%%%%%%%%%%%%%%%%%%%%%%%%%%%%%%%
	\section{\texorpdfstring{A Sz.-Nagy--Foias-type functional model}{A Sz.-Nagy--Foias-type functional model}}\label{S:FunctionalModel}

	\begin{comment}
		\begin{lemma}
			Let \((T_1, T_2)\) be a pair of \(q\)-commuting contractions on \(\mathcal{H}\), and let \(T = T_1 T_2\). Then a pair of bounded operators \((G_1, G_2)\) on \(\mathcal{D}_{T^*}\) satisfies the pair of equations
			
			\begin{align}\label{funda1}
				D_{T^*} G_1 D_{T^*} = T_1^* - T_2 T^* \quad \text{and} \quad D_{T^*} G_2 D_{T^*} = T_2^* - q T_1 T^*
			\end{align}
			
			if and only if it satisfies the following pair of equations:
			
			\begin{align}\label{funda2}
				D_{T^*} T_1^* = G_1 D_{T^*} + G_2^* D_{T^*} T^* \quad \text{and} \quad D_{T^*} T_2^* = G_2 D_{T^*} + q G_1^* D_{T^*} T^*.
			\end{align}
		\end{lemma}
		
		\begin{proof}
			
			Let the pair \((G_1, G_2)\) satisfy the equations in \eqref{funda1}. Since \(D_{T^*}\) is a bounded and injective operator on \(\mathcal{D}_{T^*}\), applying \(D_{T^*}\) to the left-hand side of the equations in \eqref{funda2} preserves the set of solutions.
			
			\begin{align*}
				D_{T^*}G_1 D_{T^*} +D_{T^*} G_2^* D_{T^*} T^*=T^*_1-T_2T^*+(T_2-\oqTT^*_1)T^*&=T^*_1-T_2T^*+T_2T^*-TT^*T^*_1\\
				&=(I-TT^*)T^*_1\\
				&=D^2_{T^*}T^*_1
			\end{align*}

			\begin{align*}
				D_{T^*}G_2 D_{T^*} +q D_{T^*} G_1^* D_{T^*} T^*=T^*_2-q T_1T^*+q (T_1-TT^*_2)T^*&=T^*_2-q T_1T^*+q T_1T^*-TT^*T^*_2\\
				&=(I-TT^*)T^*_2\\
				&=D^2_{T^*}T^*_2
			\end{align*}
			
			Therefore, the pair $(G_1,G_2)$ satisfies the system of equations presented in \eqref{funda2}.

		\end{proof}
	\end{comment}
	%%%%%%%%%%%%%%%%%%%%%%%%%%%%%%%%%%%%%%%%%%%%%%%%%%%%%%%%%%%%%%%%%%%%%%

Let $T$ be a cnu contraction acting on a Hilbert space $\cH$, $\cK_{\Theta_T}$ be the space as in \eqref{KthetaT}, and $\Pi_{\rm{NF}}:\cH\to\cK_{\Theta_T}$ be the isometry as in \eqref{NF-IsoEmbed}. It was shown in \cite[Chapter VI]{Nagy-Foias} (see also \cite[Chapter 2]{BS-CUP}) that
\begin{align}\label{HthetaT}
\operatorname{Ran}\Pi_{\rm{NF}}= \begin{bmatrix}
H^2(\cD_{T^*})\\
\overline{\Delta_{\Theta_T } L^2(\cD_T)}
\end{bmatrix}
\ominus
\begin{bmatrix}
\Theta_T \\
\Delta_{\Theta_T}
\end{bmatrix}\cdot H^2(\cD_T)=:\cH_{\Theta_T}.
\end{align}Therefore, what follows is the functional model for a cnu contraction operator due to Sz.-Nagy and Foias. Indeed, a direct consequence of the intertwining \eqref{NF-Int}, it follows that the cnu contraction $T$ is unitarily equivalent to
$$
P_{\cH_{\Theta_T}}\begin{bmatrix}
M_z&0\\
0& M_\zeta
\end{bmatrix}|_{\cH_{\Theta_T}}
$$via the unitary operator $\Pi_{\rm{NF}}:\cH\to \cH_{\Theta_T}=\operatorname{Ran}\Pi_{\rm{NF}}$.

We shall denote by $(\cD,\cD_*,\Theta)$ a bounded analytic function $\Theta:\mathbb D\to\cB(\cD,\cD_*)$. The function $(\cD,\cD_*,\Theta)$ is said to be \textit{purely contractive } if $\Theta(z)$ is a contraction for every $z\in \bD$ and furthermore, $\|\Theta(0)f\|<\|f\|$ for every non-zero vector $f$. Given a purely contractive analytic function $(\cD,\cD_*,\Theta)$, consider the space
\begin{align}\label{Ktheta}
\cH_\Theta=\begin{bmatrix}
H^2(\cD_*)\\
\overline{\Delta_{\Theta} L^2(\cD)}
\end{bmatrix}
\ominus
\begin{bmatrix}
\Theta\\
\Delta_{\Theta}
\end{bmatrix}\cdot H^2(\cD),
\end{align}where $\Delta_\Theta$ is the same operator as $\Delta_{\Theta_T}$ with $\Theta_T$ replaced by $\Theta$. Then with the contraction operator $T_\Theta$ on $\cH_\Theta$ defined by
\begin{align}
T_\Theta=P_{\cH_{\Theta}}\begin{bmatrix}
M_z&0\\
0& M_\zeta
\end{bmatrix}|_{\cH_{\Theta}},
\end{align}
it is known (see \cite[Theorem VI.3.1]{Nagy-Foias}) that the characteristic function $\Theta_{T_\Theta}$ of the contraction $T_\Theta$ coincides with $\Theta$ in the sense described in \eqref{CharcCoin}. The goal of this section is to develop a parallel theory for $q$-commuting contractive pairs. This develops in steps.

\subsection{The fundamental operators}
A key role in obtaining the functional model is played by the operators guaranteed to exist in the result below.
	\begin{thm} \label{prof funda}
		Given a pair \((T_1, T_2)\) of \(q\)-commuting contractions, there is an (in general non-commuting) operator pair $(G_1,G_2)$ acting on $\cD_{T^*}$ so that the following operator equations are satisfied: 
		\begin{align} \label{funeq}
			D_{T^*} G_1 D_{T^*} = T^*_1 - T_2 T^* \quad \text{and} \quad D_{T^*} G_2 D_{T^*} = T^*_2 -q T_1T^*.
		\end{align}
Moreover, if $(\cF_*, \Lambda_*, P_*, U_*)$ is an And\^o tuple for $(T_1^*,T_2^*)$, then
		\begin{align}  \label{funda} 
			(G_1, G_2) := (\Lambda_*^* P_*^\perp U_* \Lambda_*, \Lambda_*^* U_*^* P_* \Lambda_*).
		\end{align}
\end{thm}
		\begin{proof}		
			Once the existence is ensured, the uniqueness of the solutions $G_1,G_2$ is trivial. We prove the rest of the assertion in one stroke. Suppose $(\cF_*, \Lambda_*, P_*, U_*)$ is a special And\^o tuple for $(T_1^*,T_2^*)$, which is shown to exist in the discussion preceding Definition \ref{D:SpecialAndo}. Let $G_1,G_2$ be as in \eqref{funda}. Then for every $h,h'\in \cH$, we compute
			\begin{align*}
				\langle D_{T^*} G_1D_{T^*}h,h'\rangle=\langle P_*^\perp U_* \Lambda_* D_{T^*}h, \Lambda_* D_{T^*} h'\rangle
				&=\langle 0\oplus D_{T^*_2}T^*_1h,D_{T^*_1}T^*_2h'\oplus D_{T^*_2}h'\rangle \\
				&=\langle D^2_{T^*_2}T^*_1h,h'\rangle \\
				&=\langle (T^*_1 -T_2 T^*)h,h'\rangle.
			\end{align*}
			A similar computation, utilizing the properties of the And\^o tuple, establishes the second equality.
		\end{proof}
Considering the heavy role of the operators $G_1,G_2$ in what follows, it is convenient to give them a name for quick reference.
	\begin{definition}\label{D:Fund}
		The unique pair \((G_1, G_2)\) associated to a $q$-commuting contractive pair as in Theorem \ref{prof funda} will be called the \textit{fundamental operators} for \((T^*_1, T^*_2)\).
	\end{definition}
The first main result of this section is as follows.
\begin{thm}\label{T:FunctMod}
Let $(T_1,T_2)$ be a $q$-commuting contractive pair with $(G_1,G_2)$ the fundamental operator pair for $(T_1^*,T_2^*)$. Let $T=T_1T_2$ be a cnu contraction and $\cH_{\Theta_T}$ be the space as in \eqref{HthetaT}. Then $(T_1,T_2)$ is unitarily equivalent to
$$
P_{\cH_{\Theta_T}}\left( 
\begin{bmatrix}
			M_{G_1^* + z G_2}R_q&0\\ 0&W^{\textrm{NF}}_{1}
			\end{bmatrix},
			\begin{bmatrix}
			R_{\oq}M_{G_2^*+zG_1}&0\\
			0&W^{\textrm{NF}}_{2}
			\end{bmatrix}
\right)|_{\cH_{\Theta_T}}
$$via the map $\Pi_{\rm NF}:\cH\to \operatorname{Ran}\Pi_{\rm NF}$, where $\Pi_{\rm NF}$ is as in \eqref{NF-IsoEmbed}. 
\end{thm}
\begin{proof} 
Let us consider the operator
		\[
		U_{\textrm{NF}} =\begin{bmatrix} I_{H^2(\mathcal{D}_{T^*})} & 0 \\ 0 &  \omega_{\rm{D,NF}}\end{bmatrix}\begin{bmatrix}H^2(\mathcal{D}_{T^*}) \\ \cQ_{T^*} \end{bmatrix} \to \begin{bmatrix} H^2(\mathcal{D}_{T^*}) \\ \overline{\Delta_{\Theta_T} L^2(\mathcal{D}_T)} \end{bmatrix},
		\]
where the unitary $\omega_{\rm{D,NF}}$ is as discussed in the paragraph leading to \eqref{omega-DNF}. Then, \( U_{\textrm{NF}} \) is a unitary operator and satisfies the following relations:
		
		\begin{align}\label{DNF}
			U_{\textrm{NF}} V_D = V_{\textrm{NF}} U_{\textrm{NF}} \quad \text{and} \quad U_{\textrm{NF}} \Pi_D = \Pi_{\text{NF}}.
		\end{align}
		In accordance with the definition of \( (G_1, G_2) \) provided in \eqref{funeq}, Corollary  \ref{intphiD} yields the following intertwining relations:
		
		\begin{align}\label{PhiD}
			\Pi_D T_1^* = \begin{bmatrix}  M_{G^*_1 + zG_2} R_q & 0\\ 0& W_1 \end{bmatrix}^* \Pi_D \quad \text{and}\quad \Pi_D T_2^* = \begin{bmatrix} R_{\oq} M_{G^*_2 + zG_1} &0\\0& W_1 \end{bmatrix} ^* \Pi_D.
		\end{align}
		
		Then the following intertwining relations are derived by applying \eqref{DNF} and \eqref{PhiD}.
		
		\begin{align} \label{PhiNF}
			\Pi_{\text{NF}}T^*_1=\begin{bmatrix} M_{G^*_1 + zG_2} R_q &0\\0& W^{\textrm{NF}}_1 \end{bmatrix}^* \Pi_{\text{NF}} \quad \text{and} \quad  \Pi_{\text{NF}}T^*_2=\begin{bmatrix} R_{\oq} M_{G^*_2 + zG_1} &0\\0& W^{\textrm{NF}}_2 \end{bmatrix}^* \Pi_{\text{NF}}.
		\end{align}
		
		The relations in \eqref{PhiNF} complete the proof of the theorem.  
\end{proof}

\subsection{Pseudo $q$-commuting contractive lift}
One major advantage of the isometric lift for a single contraction is that minimal isometric lifts are unique up to isomorphism. In contrast to the classical case, a pair of \(q\)-commuting contractions may have two minimal isometric lifts that are not unitarily equivalent. The following example highlights this difficulty.
\begin{example}\label{nonisolifts}
Let \((T_1,T_2)=(0,0)\) on \(\mathbb{C}\). Then consider two pairs of \(q\)-commuting isometries \((R_qM_z, M_z)\) and \((R_{q}M_{z_1}, M_{z_2})\) on \(H^2(\mathbb{D})\) and \(H^2(\mathbb{D}^2)\), respectively. Then it is easy to see that both the pairs are minimal \(q\)-commuting isometric lifts of \((T_1,T_2)\), where the corresponding standard inclusion serves as the isometric embedding. However, they are not unitarily equivalent, since the second pair is doubly \(q\)-commuting i.e., $M_{z_2} (R_qM_{z_1})^*=q (R_qM_{z_1})^*M_{z_2}$, whereas the first pair is not.

\end{example}
To remedy this situation, we introduce the notion of a pseudo \(q\)-commuting contractive lift. It then follows that this notion of isometric lift serves as a correct analogue to the classical case, as we show that any two pseudo \(q\)-commuting contractive lifts of a pair of \(q\)-commuting contractions are unitarily equivalent. This will be useful in developing the functional model.
	\begin{definition}\label{pseudo triple}
		Let \(({W}_1, {W}_2, {W})\) be a triple of operators on a Hilbert space \(\cK\). We say that \(({W}_1, {W}_2, {W})\) is a \textit{pseudo \(q\)-commuting contractive operator triple} if:
		
		\begin{enumerate}[label=\roman*)]
			\item \({W}_1\) and \({W}_2\) are contractions, and \({W}\) is an isometry.
			\item \(({W}_1, {W})\) and \(({W}_2, {W})\) are \(q\)-commuting and \(\oq\)-commuting, respectively.
			\item \({W}_1 = \oq {W}_2^* {W}\).
		\end{enumerate}
		Furthermore, we say that a pseudo $q$-commuting contractive triple $({W}_1, {W}_2, {W})$ is a  lift of $(T_1, T_2, T_1T_2)$ via an isometric embedding $\Pi:\cH\to\cK$ if
		\begin{enumerate}[label=\roman*)]
			\item $(\Pi, {W}_1, {W}_2, {W})$ is a lift of $(T_1, T_2, T = T_1 T_2)$, i.e., 
			\[
			\Pi(T_1^*, T_2^*, T_2^* T_1^*) = ({W}_1^*, {W}_2^*, {W}^*) \Pi.
			\]
			
			\item $(\Pi, {W})$ is a minimal isometric lift for $T = T_1 T_2$, i.e.,
			\[
			\cK = \bigvee_{n \geq 0} {W}^n \operatorname{Ran} \Pi.
			\]
		\end{enumerate}
	\end{definition}
	
	%%%%%%%%%%%%%%%%%%%%%%%%%%%%%%%%%%%%%%%%%%%%%%%%%%%%%%%%%%%%%%%
	\begin{thm}[Douglas-model pseudo $q$-commuting contractive lift]\label{D:mod unique lift}
		Let \((T_1, T_2)\) be a pair of \(q\)-commuting contractions on a Hilbert space \(\mathcal{H}\). Consider the Douglas isometric-lift model space \(\mathcal{K}_D = \begin{bmatrix} H^2(\mathcal{D}_{T^*}) \\ \cQ_{T^*} \end{bmatrix}\) for \(T\) and the Douglas isometric embedding \(\Pi_D = \begin{bmatrix} \cO_{\mathcal{D}_{T^*}, T^*} \\ Q_{T^*} \end{bmatrix}\). Define the operators \(\mathbb{W}^D_1\), \(\mathbb{W}^D_2\), and \(V_D\) on \(\mathcal{K}_D\) as follows:
		
		\[
		(\mathbb{W}^D_1, \mathbb{W}^D_2, V_D) = \left( \begin{bmatrix} M_{G_1^* + z G_2} R_q & 0 \\ 0 & W_1 \end{bmatrix}, 
		\begin{bmatrix} R_{\oq} M_{G_2^* + z G_1} & 0 \\ 0 & W_2 \end{bmatrix}, 
		\begin{bmatrix} M_z & 0 \\ 0 & W_D \end{bmatrix} \right),
		\]
		where \((G_1, G_2)\) are the fundamental operators of \((T_1^*, T_2^*)\), and \((W_1, W_2)\) is the canonical \(q\)-commuting unitary pair. Then, \((\Pi_D, \mathbb{W}^D_1, \mathbb{W}^D_2, V_D)\) is a pseudo \(q\)-commuting contractive lift of \((T_1, T_2)\). Furthermore, any pseudo \(q\)-commuting contractive lift of the form  \((\Pi_D, \mathbb{W}_1, \mathbb{W}_2, V_D)\) satisfies the property that \((\mathbb{W}_1, \mathbb{W}_2) = (\mathbb{W}^D_1, \mathbb{W}^D_2)\).

	\end{thm}
	%%%%%%%%%%%%%%%%%%%%%%%%%%%%%%%%%%%%%%%%%%%%%%%%%%%%%%%%%%%%%%%%
	\begin{proof}
		We begin by showing that \((\mathbb{W}^D_1, \mathbb{W}^D_2, V_D)\) forms a pseudo \(q\)-commuting contractive triple on \(\cK_D\). It is evident that \(\mathbb{W}^D_1\) and \(\mathbb{W}^D_2\) are contractions, while \(V_D\) is an isometry. The $q$-commutativity of the pair $(\mathbb{W}^D_1, V_D)$ holds as follows:
		\begin{align*}
			\mathbb{W}^D_1 V_D=\begin{bmatrix}  M_{G_1^* + z G_2} R_q & 0 \\ 0 & W_1 \end{bmatrix} \begin{bmatrix} M_z & 0 \\ 0 & W_D \end{bmatrix} &= \begin{bmatrix}  M_{G_1^* + z G_2} R_q M_z& 0 \\ 0 & W_1W_D \end{bmatrix}\\
			&= q \begin{bmatrix}  M_{G_1^* + z G_2}M_z R_q & 0 \\ 0 & W_DW_1 \end{bmatrix}\\
			&=q\begin{bmatrix} M_z & 0 \\ 0 & W_D \end{bmatrix}\begin{bmatrix}  M_{G_1^* + z G_2} R_q & 0 \\ 0 & W_1 \end{bmatrix}\\
			&=q V_D\mathbb{W}^D_1.
		\end{align*}
		
		A similar computation shows that \((\mathbb{W}^D_2, V_D)\) is \(\oq\)-commuting. The third condition of the definition can be derived as follows:
		\begin{align*}
			\mathbb{W}^{D*}_2V_D=\begin{bmatrix} R_{\oq} M_{G_2^* + z G_1} & 0 \\ 0 & W_2 \end{bmatrix}^* \begin{bmatrix} M_z & 0 \\ 0 & W_D \end{bmatrix} &= \begin{bmatrix}  M^*_{G_2^* + z G_1}R_q M_z& 0 \\ 0 & W^*_2 W_D \end{bmatrix}\\
			&=  \begin{bmatrix}  q M^*_{G_2^* + z G_1}M_z R_q & 0 \\ 0 & W^*_2W_1W_2 \end{bmatrix}\\
			&=q\begin{bmatrix}  M_{G_1^* + z G_2} R_q & 0 \\ 0 & W_1 \end{bmatrix}\\
			&=q \mathbb{W}^D_1.
		\end{align*}
		
		Hence, \((\mathbb{W}^D_1, \mathbb{W}^D_2, V_D)\) is a pseudo \(q\)-commuting contractive triple on \(\cK_D\). Then it follows from Corollary \ref{intphiD} that \((\Pi_D, \mathbb{W}^D_1, \mathbb{W}^D_2, V_D)\) is a pseudo \(q\)-commuting contractive lift of \((T_1, T_2)\).

		To prove the second part, let \((\Pi_D, \mathbb{W}_1, \mathbb{W}_2, V_D)\) be a pseudo \(q\)-commuting contractive lift of \((T_1, T_2)\). By definition, \((\mathbb{W}_1, \mathbb{W}_2, V_D)\) is a pseudo \(q\)-commuting contractive triple on \(\mathcal{K}_D\). We express \(\mathbb{W}_1\) and \(\mathbb{W}_2\) as \(2 \times 2\) block matrices with respect to the decomposition \(\mathcal{K}_D = H^2(\mathcal{D}_{T^*}) \oplus \cQ_{T^*}\), as follows:
		
		\[
		\mathbb{W}_j = 
		\begin{bmatrix} 
			\mathbb{W}_{j,11} & \mathbb{W}_{j,12} \\
			\mathbb{W}_{j,21} & \mathbb{W}_{j,22}
		\end{bmatrix}, 
		\quad \text{for } j = 1, 2.
		\]
		
		From Definition~\ref{pseudo triple} of a pseudo \(q\)-commuting contractive triple, we derive the following:
		
		The \(q\)-commutativity of the pair \((\mathbb{W}_{1}, V_D)\) gives the relation  
		\begin{align} \label{pseudo2}
			\begin{bmatrix}
				\mathbb{W}_{1,11}M_z & \mathbb{W}_{1,12}W_D \\
				\mathbb{W}_{1,21}M_z & \mathbb{W}_{1,22}W_D
			\end{bmatrix}
			= q
			\begin{bmatrix}
				M_z \mathbb{W}_{1,11} & M_z \mathbb{W}_{1,12} \\
				W_D \mathbb{W}_{1,21} & W_D \mathbb{W}_{1,22}
			\end{bmatrix}.
		\end{align}
		
		Focusing on the \((1,2)\)-entry of this matrix identity, we obtain  
		\[
		\mathbb{W}_{1,12} \oq W_D = M_z \mathbb{W}_{1,12}.
		\]  
		This shows that the bounded operator \(\mathbb{W}_{1,12}\) intertwines the unitary operator \(\oq W_D\) with the unilateral shift \(M_z\). Then by Lemma \ref{L:Zero}, it follows that \(\mathbb{W}_{1,12} = 0\). An analogous argument, using the \(\oq\)-commutativity condition of the pair $(\mathbb{W}_2,V_D)$, yields \(\mathbb{W}_{2,12} = 0\).

		Applying the third condition, \(\mathbb{W}_1 = \oq \mathbb{W}^*_2 V_D\), from Definition~\ref{pseudo triple}, we obtain:
		
		\begin{align}\label{pseudo3}
			\begin{bmatrix}
				\mathbb{W}_{1,11} & 0 \\
				\mathbb{W}_{1,21} & \mathbb{W}_{1,22}
			\end{bmatrix}
			= \oq
			\begin{bmatrix}
				\mathbb{W}^*_{2,11} M_z & \mathbb{W}^*_{2,21} W_D \\
				0 & \mathbb{W}^*_{2,22} W_D
			\end{bmatrix}.
		\end{align}
		
		Upon compairing the entries at $(2,1)$ and $(1,2)$, we conclude that $\mathbb{W}_{1,21}=0$ and $\mathbb{W}_{2,21}=0$, respectively. Hence, the matrix \(\mathbb{W}_j\) simplifies to the following diagonal form:
		
		\[
		\mathbb{W}_j = 
		\begin{bmatrix} 
			\mathbb{W}_{j,11} & 0 \\
			0 & \mathbb{W}_{j,22}
		\end{bmatrix}, 
		\quad \text{for } j = 1, 2.
		\]
		
		Examining the \((1,1)\)-entry of the matrix equation \eqref{pseudo2}, we arrive at the relation:  
		\[
		\mathbb{W}_{1,11} M_z = q M_z \mathbb{W}_{1,11}.
		\]
		
		A corresponding matrix equation, derived from the \(\oq\)-commutativity condition for the pair \((\mathbb{W}_2, V_D)\), yields a similar identity:  
		\[
		\mathbb{W}_{2,11} M_z = \oq M_z \mathbb{W}_{2,11}.
		\]

		As a consequence of  Proposition \ref{q-commutant}, the block \( W_{j,11} \) can be expressed in the specific forms \( \mathbb{W}_{1,11} = M_{\varphi_1} R_q \) and \( \mathbb{W}_{2,11} = M_{\varphi_2} R_{\oq} \), where \( \varphi_1, \varphi_2 \in H^{\infty}(\mathcal{B}(\mathcal{D}_{T^*})) \). Let us consider the Taylor expansion of the contractive analytic functions $\varphi_j$ on $\mathbb{D}$ as 
		\[
		\varphi_j(z) = \sum_{k \ge 0} \varphi_{j,k} z^k \quad \text{for } j = 1, 2,
		\]  
		where each \(\varphi_{j,k}\) is a bounded operator on \(\mathcal{D}_{T^*}\). Comparing the \((1,1)\)-entry of \eqref{pseudo3}, we obtain \(\mathbb{W}_{1,11} = \oq \mathbb{W}_{2,11}^* M_z\), which in turn implies \(M_{\varphi_1} = \oq R_q M_{\varphi_2}^* M_z\) and the resulting relations among the Taylor coefficients are obtained as follows:
		\begin{align}\label{linpencil}
			\sum_{k \ge 0} \varphi_{1,k} z^k=M_{\varphi_1}(1) = \oq R_q M_{\varphi_2}^* M_z(1) 
			= \oq R_q \sum_{k \ge 0} \varphi^*_{2,k} \bar{z}^{(k-1)}=\sum_{k \ge 0} \varphi^*_{2,k} \oq^{k}\bar{z}^{(k-1)}.
		\end{align}
		
		By equating the coefficients in \eqref{linpencil}, we find that \(\varphi_{1,0}=\oq\varphi^*_{2,1}\), 
		\(\varphi_{1,1}=\varphi^*_{2,0}\) and \(\varphi_{j,k} = 0\) for \(j = 1, 2\) and for all \(k \geq 3\). Let us define a pair $(G_1,G_2)$ of bounded operators on $\cD_{T^*}$ by $G_1=\varphi^*_{1,0}$ and $G_2=\varphi_{1,1}$. Then the pair \((\varphi_1(z), \varphi_2(z))\) takes the form \((G_1^* + zG_2,\; G_2^* + \oq zG_1)\), which leads to the following form of \(\mathbb{W}_{1,11}\) and \(\mathbb{W}_{2,11}\):

		\begin{align}\label{phi1,phi2}
			\begin{cases}
				\mathbb{W}_{1,11} = M_{\varphi_1} R_q = M_{G^*_1 + zG_2} R_q, \\
				\begin{aligned}
					\mathbb{W}_{2,11} = M_{\varphi_2} R_{\oq} = M_{G^*_2 + \oq z G_1} R_{\oq} 
					= M_{\oq z} R_{\oq} \otimes G_1 \oplus R_{\oq} \otimes G^*_2 
					&= R_{\oq} M_z \otimes G_1 \oplus R_{\oq} \otimes G^*_2 \\
					&= R_{\oq} M_{G^*_2 + zG_1}.
				\end{aligned}
			\end{cases}\
		\end{align}

		To conclude the proof, we show that the pair \((G_1, G_2)\) consists of the fundamental operators, and \((\mathbb{W}_{1,22}, \mathbb{W}_{2,22})\) is the canonical \(q\)-commuting unitary pair for \((T_1^*, T_2^*)\).

		Since \((\Pi_D, \mathbb{W}_1, \mathbb{W}_2, V_D)\) is a pseudo \(q\)-commuting contractive lift of \((T_1, T_2, T)\), we have the following:
		
		\begin{align}\label{psedolift}
			{(\mathbb{W}_1,\mathbb{W}_2,V_D)}^* \Pi_D = \Pi_D (T_1, T_2,T)^*.
		\end{align}
		
		By combining the results of equations \eqref{phi1,phi2} and \eqref{psedolift}, along with the pseudo \( q \)-commutativity of the triple \( (\mathbb{W}_1, \mathbb{W}_2, V_D) \), we derive the following expression:

		\begin{align}\label{pseudofunda1}
			T^*_1-T_2T^*=\Pi^*_D\Pi_D(T^*_1-T_2T^*)
			&=\Pi^*_D \mathbb{W}^*_1\Pi_D-T_2\Pi^*_D\Pi_DT^*\\
			&=\Pi^*_D \mathbb{W}^*_1\Pi_D-\Pi^*_D\mathbb{W}_2V^*_D\Pi_D\notag\\
			&=\Pi^*_D( \mathbb{W}^*_1-\mathbb{W}^*_1 V_DV^*_D)\Pi_D\notag\\
			&=\Pi^*_D\mathbb{W}^*_1 (I- V_DV^*_D)\Pi_D\notag\\
			&=\Pi^*_D \begin{bmatrix} R_{\oq} M^*_{G^*_1 + zG_2}(I-M_zM^*_z) &0\\0&  0  \end{bmatrix}\Pi_D\notag\\
			&=\Pi^*_D \begin{bmatrix} R_{\oq}(I-M_zM^*_z)\otimes G_1 &0\\0&  0  \end{bmatrix}\Pi_D\notag\\
			&=\cO^*_{D_{T^*},T^*}(R_{\oq}(I-M_zM^*_z)\otimes G_1)\cO_{D_{T^*},T^*}.\notag
		\end{align}

		Based on the definitions of the observability operator \(\cO_{D_{T^*},T^*}\) and its adjoint, which are given by  
		\(
		\cO_{D_{T^*},T^*}h = \sum_{n \geq 0} z^n D_{T^*} {T^*}^n h,
		\) 
		and  
		\(
		\mathcal{O}^*_{D_{T^*},T^*} \left( \sum_{n \geq 0} h_n z^n \right) = \sum_{n \geq 0} T^n D_{T^*} h_n,
		\) 
		equation \eqref{pseudofunda1} can be rewritten in the following form:

		\begin{align}\label{pseudofunda2}
			(T^*_1-T_2T^*)h&=\cO^*_{D_{T^*},T^*}(R_{\oq}(I-M_zM^*_z)\otimes G_1)\cO_{D_{T^*},T^*}h\\
			&=\cO^*_{D_{T^*},T^*}(R_{\oq}(I-M_zM^*_z)\otimes G_1)\sum_{n\ge 0}z^n D_{T^*}{T^*}^nh\notag\\
			&=\cO^*_{D_{T^*},T^*} G_1 D_{T^*}h= D_{T^*}G_1 D_{T^*}h.\notag
		\end{align}

		%%%%%%%%%%%%%%%%%%%%%%%%%%%%%%%%%%%%%%%%%%%%%%%%%%%%%%%%%%%%%%%%%
		By a similar line of reasoning, we arrive at the following equation:
		
		\begin{align}\label{pseudofunda3}
			T^*_2 - qT_1T^* %=\Pi^*_D\Pi_D(T^*_2 - qT_1T^*)
			%&= \Pi^*_D \mathbb{W}^*_2\Pi_D - qT_1\Pi^*_D\Pi_D T^* \\
			%&= \Pi^*_D \mathbb{W}^*_2\Pi_D - q\Pi^*_D\mathbb{W}_1 V^*_D \Pi_D \notag \\
			%&= \Pi^*_D(\mathbb{W}^*_2 - \mathbb{W}^*_2 V_D V^*_D)\Pi_D \notag \\
			%&= \Pi^*_D \mathbb{W}^*_2 (I - V_D V^*_D)\Pi_D \notag \\
			%&= \Pi^*_D 
			%\begin{bmatrix} 
			%R_q M^*_{G^*_2 + zG_1}(I - M_z M^*_z) & 0 \\ 
			%0 & 0 
			%\end{bmatrix} \Pi_D \notag \\
			%&= \Pi^*_D 
			%\begin{bmatrix} 
			%{R_q (I - M_z M^*_z) \otimes G_2 & 0 \\ 
				%0 & 0
				%\end{bmatrix} \Pi_D \notag \\
				%&= \mathcal{O}^*_{D_{T^*},T^*}(R_q (I - M_z M^*_z) \otimes G_2) \cO_{D_{T^*},T^*}^* \notag \\
				&= D_{T^*} G_2 D_{T^*}.
			\end{align}
			
			%%%%%%%%%%%%%%%%%%%%%%%%%%%%%%%%%%%%%%%%%%%%%%%%%%%%%%%%%%%%%%%
			Since \((G_1, G_2)\) satisfy the fundamental equations \eqref{funeq}, as shown by \eqref{pseudofunda2} and \eqref{pseudofunda3}, the uniqueness of the fundamental operators pair implies that \((G_1, G_2)\) is the fundamental operators pair of \((T_1^*, T_2^*)\). What remains to be shown is that \( (W_{1,22}, W_{2,22}) \) indeed forms the canonical \( q \)-commuting unitary pair associated with \( (T_1^*, T_2^*) \). It follows from \eqref{psedolift} that \( W^*_{j,22} Q_{T^*} = Q_{T^*} T^*_j \) for \( j = 1, 2 \), and \( W^*_D Q_{T^*} = Q_{T^*} T^* \). By applying \eqref{pseudo2} and \eqref{pseudo3}, we deduce that \((W_{1,22}, W_{2,22}, W_D)\) constitutes a pseudo \(q\)-commuting contractive triple on \(\cQ_{T^*}\).

			\begin{align*}W^*_{1,22} W^*_{2,22} W^n_D Q_{T^*}=q ^n W^*_{1,22} W^n_D W^*_{2,22}Q_{T^*}=W^n_D W^*_{1,22} W^*_{2,22}Q_{T^*}
				&=W^n_D W^*_{1,22}Q_{T^*}T^*_2\\
				&=W^n_DQ_{T^*}T^*_1T^*_2\\
				&=q W^n_DQ_{T^*}T^*\\
				&=qW_D^n W^*_DQ_{T^*}\\
				&=q W^*_D W^n_D Q_{T^*}.
			\end{align*}
			
			Since \(\cQ_{T^*}=\overline{\text{span}}\{W^n_DQ_{T^*}h : h\in \cH\}\), we find that \(W^*_{1,22} W^*_{2,22}=qW^*_D\) on \(\cQ_{T^*}\). Thus, the result follows directly from Corollary \ref{uniquecano}.
			
		\end{proof}
		
		The above theorem leads to the following corollary, which establishes the uniqueness of the pseudo \(q\)-commuting contractive lift of a pair \((T_1, T_2)\) of \(q\)-commuting contractions.
		
		%%%%%%%%%%%%%%%%%%%%%%%%%%%%%%%%%%%%%%%%%%%%%%%%%%%%%%%%%%%%%%%
		\begin{corollary}\label{D-mod pseudo q-commuting}
Let \((T_1, T_2)\) be a \(q\)-commuting contractive pair acting on a Hilbert space \(\mathcal{H}\). Suppose \((\Pi, \mathbb{W}_1, \mathbb{W}_2, \mathbb{W})\) is a pseudo-$q$-commuting contractive lift of $(T_1,T_2,T_1T_2)$. Then\((\Pi, \mathbb{W}_1, \mathbb{W}_2, \mathbb{W})\) is  unitarily equivalent to the Douglas-model pseudo \(q\)-commuting contractive lift
\begin{align}\label{AsInLift}
\left(\Pi_D, \begin{bmatrix} M_{G_1^* + z G_2} R_q & 0 \\ 0 & W_1 \end{bmatrix}, \begin{bmatrix} R_{\oq} M_{G_2^* + z G_1} & 0 \\ 0 & W_1 \end{bmatrix}, \begin{bmatrix} M_z & 0 \\ 0 & W_D \end{bmatrix} \right).
\end{align}
		\end{corollary}
		
		\begin{proof}
			Since \((\Pi, \mathbb{W}_1, \mathbb{W}_2, \mathbb{W})\) is a pseudo \(q\)-commuting contractive lift, by definition, \((\Pi, \mathbb{W})\) is a minimal isometric lift of \(T = T_1 T_2\) on \(\mathcal{K}\). By the uniqueness of the minimal isometric lift of \(T\), there exists a unitary \(\tau: \mathcal{K} \to \mathcal{K}_D\) such that \(\tau \Pi = \Pi_D\) and \(\tau \mathbb{W} = V_D \tau\). Since \((\mathbb{W}_1, \mathbb{W}_2, \mathbb{W})\) is a pseudo \(q\)-commuting contractive triple and \(\tau\) is unitary, \(\tau (\mathbb{W}_1, \mathbb{W}_2, \mathbb{W}) \tau^*\) is also a pseudo \(q\)-commuting contractive triple. Again, since \((\Pi, \mathbb{W}_1, \mathbb{W}_2, \mathbb{W})\) is a lift of \((T_1, T_2, T = T_1 T_2)\), it follows that \[(\tau \Pi = \Pi_D, \tau \mathbb{W}_1 \tau^*, \tau \mathbb{W}_2 \tau^*, \tau \mathbb{W} \tau^* = V_D)\] is also a lift of \((T_1, T_2, T = T_1 T_2)\). By the Theorem \ref{D:mod unique lift}, we conclude that
			
			\[
			(\tau \mathbb{W}_1 \tau^*, \tau \mathbb{W}_2 \tau^* )= \left( \begin{bmatrix} M_{G_1^* + z G_2} R_q & 0 \\ 0 & W_1 \end{bmatrix}, 
			\begin{bmatrix} R_{\oq} M_{G_2^* + z G_1} & 0 \\ 0 & W_2 \end{bmatrix}
			\right).
			\]
			Hence, \(\tau\) implements a unitary equivalence between the pseudo \(q\)-commuting contractive lifts \((\Pi, \mathbb{W}_1, \mathbb{W}_2, \mathbb{W})\) and the lift as in \eqref{AsInLift}.
		\end{proof}
\subsection{The characteristic triple}\label{SS:CharcTriple}
The characteristic function, together with the fundamental operators and the canonical \(q\)-commuting unitaries provide  a complete unitary invariant for a pair of \(q\)-commuting contractions with the product being cnu.
	\begin{definition}
Given a pair \((T_1, T_2)\) of \(q\)-commuting contractions, let \(\mathbb{G} := \{G_1, G_2\}\) be the fundamental operators for $(T_1^*,T_2^*)$, and \(\mathbb{W} := \{W^{{\textrm{NF}}}_1, W^{{\textrm{NF}}}_2\}\) be the commuting unitary pair as defined in \eqref{unifunda}. The triple \((\mathbb{G}, \mathbb{W}, \Theta_T)\) is then called the {\em characteristic triple} for \((T_1, T_2)\). 
\end{definition}		
The characteristic triples for a $q$-commuting contractive pair exhibit the same characteristics as the characteristic functions of a contractive operator do. The first instance is the following.
\begin{thm}\label{T:CharcUniInv}
			Let \((T_1, T_2)\) and \((T'_1, T'_2)\) be two pairs of \(q\)-commuting contractions acting on a Hilbert space $\cH$. If \((T_1, T_2)\) and \((T'_1, T'_2)\) are unitarily equivalent, then their characteristic triples coincide. Moreover, if \(T = T_1T_2\) and \(T' = T'_1T'_2\) are cnu contractions, then the coverse also holds.
\end{thm}
		\begin{proof}
Let \(((G_1, G_2),(W_1,W_2), \Theta_T)\) and \(((G'_1, G'_2),(W'_1,W'_2), \Theta_{T'})\) be the characteristic triples of \((T_1, T_2)\) and \((T'_1, T'_2)\), respectively. To prove the forward implication, suppose \( U : \cH \to \cH' \) is a unitary operator such that \( U(T_1, T_2) = (T'_1, T'_2)U \). Then, the intertwining relations \( U D_T = D_{T'} U \) and \( U D_{T^*} = D_{T'^*} U \) imply that the operators \( u \) and \( u_* \), defined as the unitary restrictions of \( U \) to \( D_T \) and \( D_{T^*} \), respectively, satisfy the following property for all \( h \in \cH \) and $z\in\mathbb{D}$:
			
			\[
			\begin{aligned}
				\Theta_{T'}(z)\, u\, D_T h &= \restr{\left(-T' + z D_{T'^*} (I_{\cH'} - z T'^*)^{-1} D_{T'}\right)}{\cD_{T'}} u D_T h \\
				&= u_* \left(-T + z D_{T^*} (I_{\cH} - z T^*)^{-1} D_T\right) D_T h\\
				&=u_*\Theta_T(z).
			\end{aligned}
			\]
			
			Hence, \(\Theta_T\) and \(\Theta_{T'}\) coincide. The condition \((G'_1, G'_2) = u_* (G_1, G_2) u_*^*\) follows from the fundamental equations for \((T_1, T_2)\) and \((T'_1, T'_2)\) defined in \eqref{funeq}, together with the intertwining relations \( U D_T = D_{T'} U \) and \( U D_{T^*} = D_{T'^*} U \).

			%\[D_{T^*} u_*^* G'_1 u_* D_{T^*} h = u_*^* D_{T'^{*}} G'_1 D_{T'^{*}} u_* = U^* (T'^{*}_1 - T'_2 T'^{*}) = T^*_1 - T_2 T^* = D_{T^*} G_1 D_{T^*}.\] Then, Remark \ref{unifun} completes the proof that \( u_*^* G'_1 u_* = G_1 \). A similar computation shows the second equality: \( u_*^* G'_2 u_* = G_2 \).}

		%%%%%%%%%%%%%%%%%%%%%%%%%%%%%%%%%%%
		It remains to show that \((W_1, W_2)\) and \((W'_1, W'_2)\) are equivalent, implemented by the unitary \(\omega_u\). Let 
		\[\left(\Pi_{\text{NF}},\ M_{G_1^* + z G_2} R_q \oplus W_1,\ R_{\oq} M_{G_2^* + z G_1} \oplus W_2,\ M_z \oplus \left.\!M_\zeta\right|_{\overline{\Delta_{\Theta_T} L^2(\mathcal{D}_T)}}\right)\] on \(\mathcal{K}_{\theta_T}\), and  \[
		\left(\Pi'_{\text{NF}},\ M_{{G'_1}^* + z G'_2} R_q \oplus W'_1,\ R_{\oq} M_{{G'_2}^* + z G'_1} \oplus W'_2,\ M_z \oplus \left.\!M_\zeta\right|_{\overline{\Delta_{\Theta_{T'}} L^2(\mathcal{D}_{T'})}}\right)\] on \(\mathcal{K}_{\theta_{T'}}\), be the pseudo \(q\)-commuting contractive lifts of \((T_1,T_2,T)\) and \((T'_1,T'_2,T')\), respectively. Consider the isometry \(\Pi'': \mathcal{H} \to H^2(\mathcal{D}_{T^*}) \oplus \overline{\Delta_{\Theta_T} L^2(\mathcal{D}_T)}\) defined by  \[\Pi'' = \left( (I_{H^2} \otimes u^*_*) \oplus \omega^*_u \right) \Pi'_{\text{NF}} U.\]
		
		Then it follows that
		\begin{align*}
			\Big(
			\Pi'',\ 
			M_{G_1^* + z G_2} R_q \oplus W''_1,\ 
			R_{\oq} M_{G_2^* + z G_1} \oplus W''_2,\ 
			M_z \oplus \left. M_\zeta \right|_{\overline{\Delta_{\Theta_T} L^2(\mathcal{D}_T)}}
			\Big)
		\end{align*}
		is a pseudo \(q\)-commuting contractive lift of \((T_1, T_2, T)\), where \((W''_1, W''_2) = \omega_u^* (W'_1, W'_2) \omega_u\). Establishing that \(\Pi'' = \Pi_{\text{NF}}\) via the intertwining properties of \(U\) and \(\omega_u\) allows us to invoke Corollary \ref{D-mod pseudo q-commuting}, from which it follows that \(\omega_u^* (W'_1, W'_2) \omega_u = (W''_1, W''_2) = (W_1, W_2)\).
		
		%%%%%%%%%%%%%%%%%%%%%%%%%%%%%%%%%%%%%%
		
		\begin{comment}	
			\begin{align}
				\Pi''T^*_1=\begin{bmatrix} I_{H^2} \otimes u^*_*& 0\\0& \omega^*_u \end{bmatrix}\Pi'_{\text{NF}} U T^*_1
				&=\begin{bmatrix} I_{H^2} \otimes u^*_*& 0\\0& \omega^*_u \end{bmatrix}\Pi'_{\text{NF}}{T'}^*_1U\\
				&=\begin{bmatrix} I_{H^2} \otimes u^*_*& 0\\0& \omega^*_u \end{bmatrix}\begin{bmatrix} R_{\oq} M^*_{{G'_1}^* + z G'_2}  & 0\\0& {W'}^*_1 \end{bmatrix}\Pi'_{\text{NF}}U\\
				&=\begin{bmatrix} R_{\oq} M^*_{{G_1}^* + z G_2}  & 0\\0& {W''}^*_1 \end{bmatrix} \begin{bmatrix} I_{H^2} \otimes u^*_*& 0\\0& \omega^*_u \end{bmatrix} \Pi'_{\text{NF}}U\\
				&=\begin{bmatrix} R_{\oq} M^*_{{G_1}^* + z G_2}  & 0\\0& {W''}^*_1 \end{bmatrix} \Pi''_{\text{NF}}	\end{align}		
		\end{comment}

		%%%%%%%%%%%%%%%%%%%%%%%%%%%%%%%%%%
		
		Conversely, suppose that \(T, T'\) are cnu contractions and \(((G_1, G_2), (W_1, W_2), \Theta_T)\) and \(((G'_1, G'_2), (W_1, W_2), \Theta_{T'})\) coincide. Then by Theorem \ref{T:FunctMod}, \((T_1,T_2)\) and \((T'_1,T'_2)\) are unitarily equivalent to their corresponding functional model. The coincidence of characteristic triples implies that  the unitary operator 

		\[
		I_{H^2} \otimes u \oplus \omega_u : H^2(\mathcal{D}_{T^*}) \oplus \overline{\Delta_{\Theta_T} L^2(\mathcal{D}_T)} \to H^2(\mathcal{D}_{T'^*}) \oplus \overline{\Delta_{\Theta_{T'}} L^2(\mathcal{D}_{T'})}
		\]
		intertwines the operator triple
		\[
		(M_{G_1^* + z G_2} R_q \oplus W_1,\ R_{\oq} M_{G_2^* + z G_1} \oplus W_2,\ M_z \oplus \restr{M_\zeta}{\overline{\Delta_{\Theta_T} L^2(\mathcal{D}_T)}})
		\]
		with
		\[
		(M_{{G'_1}^* + z G'_2} R_q \oplus W'_1,\ R_{\oq} M_{{G'_2}^* + z G'_1} \oplus W'_2,\ M_z \oplus \restr{M_\zeta}{\overline{\Delta_{\Theta_{T'}} L^2(\mathcal{D}_{T'})}}),
		\]
	where  $u:\mathcal{D}_{T} \to \mathcal{D}_{T^*}$ and $u_*:\mathcal{D}_{T'} \to \mathcal{D}_{T'^*}$ are the unitary operators implementing the coincidence. 
		Moreover, the unitary operator \( u_* \oplus \omega_u \) maps \( \mathcal{K}_{\Theta_T} \) onto \( \mathcal{K}_{\Theta_{T'}} \), and \( (\Theta_T \oplus \Delta_{\Theta_T}) H^2(\mathcal{D}_T) \) onto \( (\Theta_{T'} \oplus \Delta_{\Theta_{T'}}) H^2(\mathcal{D}_{T'}) \). Hence, it maps \( \mathcal{H}_{\Theta_T} \) onto \( \mathcal{H}_{\Theta_{T'}} \). Thus, the functional models for \((T_1, T_2)\) and \((T'_1, T'_2)\) are unitarily equivalent. Therefore, we conclude that \((T_1, T_2)\) and \((T'_1, T'_2)\) are unitarily equivalent.
	\end{proof}
		
		%%%%%%%%%%%%%%%%%%%%%%%%%%%%%%%%%%%%%%%%%%%%%%%%%%%%%%%%%%%%%%%
	\subsection{Functional model}	
Theorem \ref{T:FunctMod} gives a unitary copy on function space of an abstract $q$-commuting contractive pair $\underline{T}=(T_1,T_2)$ in terms of its characteristic triple $\Xi_{\underline{T}}=((G_1,G_2),)(W_1^{\rm NF},W_2^{\rm NF}),\Theta_T)$. Here we explore a converse of this: We identify a triple $\Xi=((G_1,G_2),)(W_1,W_2),\Theta)$ consisting of operators $G_j$, $W_j$ for $j=1,2$ acting on appropriate spaces and a contractive analytic function $\Theta$, from which a $q$-commuting contractive pair $\underline{T}_\Xi$ can be defined so that the characteristic triple for the $q$-commuting contractive pair $\underline{T}_\Xi$ coincides with $\Xi$ in the following sense.
\begin{definition}\label{charcoincide}
Let \((\cD, \cD_*, \Theta)\) and \((\cD', \cD'_*, \Theta')\) be two purely contractive analytic functions (i.e., both functions satisfy $\|\Theta(0)f\|<\|f\|$ for every non-zero vector $f$). Let \((G_1, G_2)\) and \((G'_1, G'_2)\) be pairs of contractions on \(\cD_*\) and \(\cD'_*\), respectively. Additionally, let \((W_1, W_2)\) and \((W'_1, W'_2)\) be pairs of \(q\)-commuting unitaries, whose products are equal to \(M_{\zeta}\) on \(\overline{\Delta_{\Theta} L^2(\cD)}\) and \(\overline{\Delta_{\Theta'} L^2(\cD')}\), respectively. Then we say that the two triples \(((G_1, G_2), (W_1, W_2), \Theta)\) and \(((G'_1, G'_2), (W'_1, W'_2), \Theta')\) \textit{coincide} if the following conditions are satisfied:
			\begin{enumerate}[label=\roman*)]
				\item The pairs \((\cD, \cD_*, \Theta)\) and \((\cD', \cD'_*, \Theta')\) coincide, i.e., there exist unitary operators \(u: \cD \to \cD'\) and \(u_*: \cD_* \to \cD'_*\) such that \(u_* \Theta(z) = \Theta'(z) u\) for all \(z \in \mathbb{D}\).
				
				\item The unitary operators \(u\) and \(u_*\) satisfy the intertwining relations:
				
				\[
				(G'_1, G'_2) = u_* (G_1, G_2) u_*^* \quad \text{and} \quad (W'_1, W'_2) = \omega_u (W_1, W_2) \omega_u^*
				\]
				where \(\omega_u: \overline{\Delta_{\Theta} L^2(\cD)} \to \overline{\Delta_{\Theta'} L^2(\cD')}\) is the unitary map defined by 
				$$
				\omega_u = \restr{(I_{L^2} \otimes u)}{\overline{\Delta_{\Theta} L^2(\cD)}}.
				$$
			\end{enumerate}
		\end{definition}

		%%%%%%%%%%%%%%%%%%%%%%%%%%%%%%%%%%%%%%%%%%%%%%%%%%%%%%%%%%%%%%%%%
	%%%%%%%%%%%%%%%%%%%%%%%%%%%%%%%%%%%%%%%%%%%%%%%%%%%%%%%%%%%%%
	Next, we define those triples $\Xi=((G_1,G_2),)(W_1,W_2),\Theta)$ consisting of operators $G_j$, $W_j$ for $j=1,2$ acting on appropriate spaces and a contractive analytic function $\Theta$, which would give us the converse of Theorem \ref{T:FunctMod} alluded to in the beginning of this subsection.

	\begin{definition}\label{D:admiss}
		Let \((\mathcal{D}, \mathcal{D}_*, \Theta)\) be a contractive analytic function, and let \((G_1, G_2)\) be a pair of contractions on \(\mathcal{D}_*\). Let \((W_1, W_2)\) be a pair of \(q\)-commuting unitaries on \(\overline{\Delta_{\Theta} L^2(\mathcal{D})}\). We say that the triple \(((G_1, G_2), (W_1, W_2), \Theta)\) is \textit{admissible} if it satisfies the following admissibility conditions: 
		\begin{enumerate}
			\item \( M_{G_1^* + z G_2} R_q \oplus W_1 \) and \( R_{\oq} M_{G_2^* + z G_1} \oplus W_2 \) are contractions on \( H^2(\mathcal{D}_{*}) \oplus \overline{\Delta_{\Theta}L^2(\mathcal{D})} \).
			
			\item \( W_1 W_2 = \restr{ M_{\zeta}}{\overline{\Delta_{\Theta} L^2(\mathcal{D})}} \).
			
			\item The space \( \cQ_{\Theta} := (\Theta \oplus \Delta_{\Theta}) H^2(\mathcal{D}) \) is jointly invariant under 
			\[
			(M_{G_1^* + z G_2} R_q \oplus W_1, R_{\oq} M_{G_2^* + z G_1} \oplus W_2, M_z \oplus \restr{ M_{\zeta}}{\overline{\Delta_{\Theta} L^2(\mathcal{D})}}).
			\]
			
			\item Let \( \mathcal{K}_{\Theta} := H^2(\mathcal{D}_{*}) \oplus \overline{\Delta_{\Theta} L^2(\mathcal{D})} \) and \( \mathcal{H}_{\Theta} := \mathcal{K}_{\Theta} \ominus \cQ_{\Theta} \). Then, we have the following:
			\[
			\begin{aligned}
				\restr{q(R_{\oq} M^*_{G_1^* + z G_2} \oplus W^{*}_1 )(M^*_{G_2^* + z G_1} \oplus W^{*}_2)}{\cH_{\Theta}}
				&= \restr{(M^*_{G_2^* + z G_1} \oplus W^{*}_2) (R_{\oq} M^*_{G_1^* + z G_2} \oplus W^{*}_1)}{\cH_{\Theta}} \\
				&= \restr{(M^*_z \oplus \restr{M^*_{\zeta}}{\overline{\Delta_{\Theta} L^2(\cD)}})}{\cH_{\Theta}}.
			\end{aligned}
			\]
		\end{enumerate}
If, in addition, \(\Theta: {D} \to \cB(\cD, \cD_*)\) is purely contractive, i.e., \(\|\Theta(0)f\| < \|f\|\) for all \(f \in \cD \setminus \{0\}\), we say that \(((G_1, G_2), (W_1, W_2), \Theta)\) is \textit{purely contractive admissible triple}.
	\end{definition}

	%%%%%%%%%%%%%%%%%%%%%%%%%%%%%%%%%%%%%%%%%%%%%%%%%%%%%%%%%

	%%%%%%%%%%%%%%%%%%%%%%%%%%%%%%%%%%%%%%%%%%%%%%%%%%%%%%%%%

\begin{example}
Theorem \ref{T:FunctMod} provides us with ample examples of admissible triples, viz., the characteristic triple \(((G_1, G_2), (W_1, W_2), \Theta_T)\) of a pair of \(q\)-commuting contractions \((T_1, T_2)\) is an admissible triple. Moreover, if $T=T_1T_2=qT_2T_1$ is cnu, then the characteristic triple for $(T_1,T_2)$ is a purely contractive admissible triple because in this case $\Theta_T$ is purely contractive. 
\end{example}
	%%%%%%%%%%%%%%%%%%%%%%%%%%%%%%%%%%%%%%%%%%%%%%%%%%%%%%%%%
	\begin{thm}
		Let \(((G_1, G_2), (W_1, W_2), \Theta)\) be a purely contractive admissible triple. Consider the operator pair defined by
		\[
		(T_1, T_2) = \restr{P_{\cH_{\Theta}}(M_{G_1^* + z G_2} R_q \oplus W_1, R_{\oq} M_{G_2^* + z G_1} \oplus W_2)}{\cH_\Theta}.
		\]Then the characteristic triple of the pair \((T_1, T_2)\) coincides with \(((G_1, G_2), (W_1, W_2), \Theta)\).
	\end{thm}
	\begin{proof}
	By part (1) of Definition \ref{D:admiss}, it follows that \((T_1, T_2)\), as defined in the statement, is a pair of \(q\)-commuting contractions. Moreover, by part (4) of Definition \ref{D:admiss}, the product operator \(T := T_1 T_2\) is given by
		\[
		T = \restr{P_{\cH_{\Theta}}\left(M_z \oplus \restr{M_{\zeta}}{\overline{\Delta_{\Theta} L^2(\mathcal{D})}}\right)}{\cH_\Theta}.
		\]
	Therefore by the Sz.-Nagy--Foias model theory for single contractions (see \cite[Theorem VI.3.1]{Nagy-Foias}), \(T\) is a cnu contraction. Let \(((G'_1, G'_2), (W'_1, W'_2), \Theta_T)\) be the characteristic triple of \((T_1, T_2)\). We need to show that \(((G_1, G_2), (W_1, W_2), \Theta)\) coincides with \(((G'_1, G'_2), (W'_1, W'_2), \Theta_T)\). Since \(\Theta\) is a purely contractive analytic function, by \cite[Theorem VI.3.1]{Nagy-Foias}, \(\Theta\) coincides with the characteristic function \(\Theta_T\) of the model operator \(T\), i.e., there exist unitary operators \(u: \mathcal{D} \to \mathcal{D}_T\) and \(u_*: \mathcal{D}_* \to \mathcal{D}_{T^*}\) such that \(\Theta_T u = u_* \Theta\). Thus we already have established condition (i) of Definition \ref{charcoincide}. 
	
	To show condition (ii) of Definition \ref{charcoincide}, we proceed as follows. Let $\omega_u$ be the unitary operator corresponding to $u$ given by
$$
				\omega_u = \restr{(I_{L^2} \otimes u)}{\overline{\Delta_{\Theta} L^2(\cD)}}.
$$
Then the unitary operator \( u_* \oplus \omega_u \) maps 
\[
 \cK_{\Theta} =\begin{bmatrix}
 H^2(\cD_*)\\
\overline{\Delta_\Theta L^2(\cD)}
 \end{bmatrix}
 \] onto 
 \[ 
 \cK_{\Theta_T} =\begin{bmatrix}
 H^2(\cD_{T^*}\\
\overline{\Delta_{\Theta_T} L^2(\cD_T)}
 \end{bmatrix},
 \] and it takes the closed subspace
 \( \sbm{
   \Theta \\
    \Delta_{\Theta}
} H^2(\cD)
\) onto the closed subspace
\(
\sbm{
\Theta_T \\
 \Delta_{\Theta_T}
} H^2(\cD_T)
  \). Therefore, it maps 
  \[
   \cH_{\Theta} =
   \begin{bmatrix}
 H^2(\cD_*)\\
\overline{\Delta_\Theta L^2(\cD)}
 \end{bmatrix} \ominus
 \begin{bmatrix}
 \Theta \\
    \Delta_{\Theta}
 \end{bmatrix} H^2(\cD)
   \] onto
    \[ 
    \cH_{\Theta_T} =
    \begin{bmatrix}
 H^2(\cD_{T^*}\\
\overline{\Delta_{\Theta_T} L^2(\cD_T)}
 \end{bmatrix}\ominus
 \begin{bmatrix}
 \Theta_T \\
 \Delta_{\Theta_T}
 \end{bmatrix} H^2(\cD_T).
  \] Let \(\tau\) be the restriction of \(u_* \oplus \omega_u\) to \(\cH_{\Theta}\). Then it is an easy check that the following diagram is commutative where \(i\) and  \(i'\) are the inclusion maps:
		\[
		\begin{tikzcd}
			\cH_{\Theta} \arrow[r, "i"] \arrow[d, "\tau"'] & \cK_{\Theta} \arrow[d, "u_* \oplus  \omega_u"] \\
			\cH_{\Theta_T} \arrow[r, "i'"] & \cK_{\Theta_T}
		\end{tikzcd}.
		\]
		It is easy to check that 
		\[
		\left(i,\ M_{G^*_1 + z G_2 R_q} \oplus W_1,\ R_{\oq} M_{G^*_2 + z G_1} \oplus W_2,\ M_z \oplus \restr{M_{\zeta}}{\overline{\Delta_{\Theta} L^2(\mathcal{D})}} \right)
		\]  
		acting on \(\cK_{\Theta}\), and   
		\[
		\left(i' \circ \tau,\ M_{G'^*_1 + z G'_2 R_q} \oplus W'_1,\ R_{\oq} M_{G'^*_2 + z G'_1} \oplus W'_2,\ M_z \oplus \restr{M_{\zeta}}{\overline{\Delta_{\mathbb T} L^2(\mathcal{D}_T)}} \right)
		\]  
		acting on \(\cK_{\Theta_T}\) both are pseudo \(q\)-commuting contractive lift of \((T_1, T_2,T_1T_2)\) . By Corollary~\ref{D-mod pseudo q-commuting}, there exists a unitary operator $U : \cK_{\Theta} \to \cK_{\Theta_T}$ such that $U \circ i = i' \circ \tau$ and
		\begin{align}\label{TheInt}
			&U \left( M_{G^*_1 + z G_2 R_q} \oplus W_1,\ R_{\oq} M_{G^*_2 + z G_1} \oplus W_2,\ M_z \oplus \restr{M_{\zeta}}{\overline{\Delta_{\Theta} L^2(\mathcal{D})}} \right)\\ &= \left( M_{G'^*_1 + z G'_2 R_q} \oplus W'_1,\ R_{\oq} M_{G'^*_2 + z G'_1} \oplus W'_2,\ M_z \oplus \restr{M_{\zeta}}{\overline{\Delta_{\mathbb T} L^2(\mathcal{D}_T)}} \right) U.
		\end{align} By definition of pseudo $q$-commuting contractive lift,
		\[
		M_z \oplus \restr{M_{\zeta}}{\overline{\Delta_{\Theta} L^2(\mathcal{D})}} \quad \text{and} \quad M_z \oplus \restr{M_{\zeta}}{\overline{\Delta_{\Theta} L^2(\mathcal{D})}}
		\]  are minimal isometric lifts of $T=T_1T_2$ and therefore they are unitarily equivalent by a unique unitary similarity transformation. Since the unitary \( u_* \oplus \omega_u \) is easily seen to intertwine these two minimal isometric lifts, it follows that 
$$
U=u_* \oplus \omega_u.
$$		
This coupled with the intertwining \eqref{TheInt} gives the required condition (ii) of Definition \ref{charcoincide}. Hence, \(((G_1, G_2), (W_1, W_2), \Theta)\) coincides with \(((G'_1, G'_2), (W'_1, W'_2), \Theta_T)\). This completes the proof.
	\end{proof}
	
\section*{Acknowledgements} The author thanks his advisor, Haripada Sau, for suggesting the problem and many helpful and stimulating discussions throughout the development of this work.


\begin{thebibliography}{99}
		
		%\bibitem{AKM}J. Agler, G. Knese and J.E. McCarthy, {\em Algebraic pairs of isometries}, J. Operator Theory 67 (2012), 215-236.
		
		%\bibitem{AM_Acta} J. Agler and J.E. McCarthy, {\em Distinguished varieties}, Acta Math. 194 (2005), no.2, 133-153.
		
		%\bibitem{AMbook} J. Agler and J.E. McCarthy, Pick interpolation and Hilbert function spaces. Graduate Studies in Mathematics, 44. American Mathematical Society, Providence, RI, 2002.
		
		%\bibitem{AgMcSt} J. Agler and J.E. McCarthy, M. Stankus, {\em Toral algebraic sets and function theory on polydisks} J. Geom. Anal. 16 (2006), no. 4, 551-562.
		
		%\bibitem{Ando} T. And\^o, {\em{On a pair of commuting contractions}}, Acta Sci. Math. (Szeged) 24 (1963), 88-90.
		
		       %\bibitem{ArvesonII}W. B. Arveson, \textit{Subalgebras of $C^*$-algebra II}, Acta Math. 128 (1972), 271-308.
		
                            \bibitem{ando} T. And\^o, {\em{On a Pair of Commuting Contractions}}, Acta Sci. Math. (Szeged) \textbf{24} (1963), 88-90.
		
                           \bibitem{BS-CUP} J.A. Ball and H. Sau, \textit{Dilation and Model Theory for Pairs of Commuting Contraction Operators,} to appear in Cambridge Tracts in Mathematics.

                         % \bibitem{BS-BCL} J.A.\ Ball and H.\ Sau, {\em Commuting and doubly commuting pairs of isometries}, in: Operator Theory, Related Fields, and Applications IWOTA 2023,  15 - 67, Oper.\ Theory Adv.\ Appl. \textbf{307}, Birkh\"auser, Cham, 2025.

                          \bibitem{BS-qBCL} J.A.\ Ball and H.\ Sau, {\em Models for $q$-commuting and doubly $q$-commuting pairs of isometries}, to appear in Oper.\ Theory Adv.\ Appl. 
		
		
		%\bibitem{JB'03} J. A. Ball and V. Vinnikov, {\em Overdetermined multidimensional systems: State space and frequency domain methods} Mathematical Systems Theory in Biology, Communications, Computation, and Finance, 134, IMA Vol. Math. AopL, Springer, Berlin, (2003), 63-119.
		\bibitem{BPS} B. Bisai, S. Pal and P. Sahasrabuddhe, {\em On {$q$}-commuting co-extensions and {$q$}-commutant lifting}, Linear Algebra Appl., 658, (2023), 186--205.

                     \bibitem{BB} S. Barik, B. Bisai, {\em A generalization of {A}ndo's dilation, and isometric dilations for a class of tuples of {$q$}-commuting contractions}, Complex Anal. Oper. Theory 18, 131 (2024), 1-29.

		%\bibitem{BDF} H. Bercovici, R.G. Douglas and C. Foias, {\em On the classification of multi-isometries}, Acta Sci. Math. (Szeged) 72 (2006), 639–661.
		
		
		
                    \bibitem{BCL} C. A. Berger, L. A. Coburn and A. Lebow, {\em Representation and index theory for $C^{*}$-algebra generated by commuting isometries}, J. Functional Analysis 27 (1978), 51-99.
		
		
		%\bibitem{RB} R. Bezrukavnikov, {\em Algebraic Geometry 1 Lecture Notes}, MIT Open Courseware mumber 18.725, available at https://ocw.mit.edu/courses/mathematics/18-725-algebraic-geometry-fall-2015
		
		%\bibitem{BBK} T. Bhattacharyya, M. Bhowmik and P. Kumar, {\em Herglotz’s representation and Carath\'eodory’s approximation}, Bull. Lond. Math. Soc. 56 (2024), no. 12, 3752–3776.
		
		%\bibitem{BRV} T. Bhattacharyya, S. Rastogi and V. Kumar U., {\em The Joint Spectrum for a Commuting Pair of isometries in Certain Cases,} Complex Anal. Oper. Theory 16, 83 (2022), 1-39.
		
		%\bibitem{BKS-APDE} T. Bhattacharyya, P. Kumar, H. Sau, {\em Distinguished varieties through the Berger--Coburn--Lebow theorem}, Anal.\ PDE. Vol. 15 (2022), No. 2, 477–506.
		
                     \bibitem{BhSrPal} T. Bhattacharyya, S. Pal and S. Shyam Roy, {\em{Dilations of $\Gamma$-contractions
by solving operator equations}}, Adv. Math. \textbf{230} (2012), 577-606.
		
		%\bibitem{Bur} Z. Burdak, {\em On the model and invariant subspaces for pairs of commuting isometries}, Integral Equations Operator Theory 91 (2019), no. 3, Paper No. 22, 23 pp.
		
		%\bibitem{DKS} B. K. Das, P. Kumar and H. Sau, {\em Distinguished varieties and the Nevanlinna-Pick interpolation problem on the symmetrized bidisk}, to appear in Math. Z. arXiv:2104.12392.
		
		%\bibitem{DSJFA2017} B.K. Das and J. Sarkar, {\em Ando dilations, von Neumann inequality and distinguished varieties}, J. Functional Analysis 272 (2017), 2114-2131.
		
		%\bibitem{DSS}B.K. Das, J. Sarkar and S. Sarkar, {\em Factorizations of contractions}, Adv. Math. 322 (2017), 186–200.
		
		
		%\bibitem{DasSauPAMS} B.\ K.\ Das and H.\ Sau, {\em Pure inner functions, distinguished varieties and toral algebraic commutative contractive pairs}, Proc. Amer. Math. Soc., 152 (2024), 1067-1081.
		
		
		%\bibitem{Dix}  J.\ Dixmier, {\em von Neumann Algebras} (with a preface by E.C.\ Lance), translated from the second French edition by F.\ Jellett,North-Holland Math.\ Library \textbf{27}, North Holland Publishing Co., Amsterdam/New York, 1981.
		
                      \bibitem{Doug-Dilation} R. G. Douglas, {\em Structure theory for operators. I.}, J. Reine Angew. Math. \textbf{232} (1968) 180-193.
		
		
		
		
		%\bibitem{DSSS} S. De, P.  Shankar.  J. Sarkar, and T. R. Sankar, {\em Pairs of projections and commuting isometries}, Journal of Operator Theory 91 (2024), no. 1, 261–294. 
		
		
		%\bibitem{Knese-TAMS2010} G. Knese, {\em Polynomials defining distiguished varieties}, Trans. Amer. Math. Soc. 362 (2010), 5635-5655.
		
		%\bibitem{AL}P. Grifiths and J. Harris, {\em Principles of Algebraic Geometry}, John Wiley and Sons, 1978.
		
		%\bibitem{GWZ} K. Guo, K. Wang, and C. Zhao, {\em Essentially normal quotient weighted Bergman modules over the bidisk and distinguished varieties}, Adv. Math. 432 (2023), Paper No. 109266, 30 pp.
		
                      \bibitem{Halmos} P.R.~Halmos, {\em Shifts on Hilbert spaces}, J. Reine Angew. Math. \textbf{208} (1961) 102-112.
		
		\bibitem{Nagy-Foias} B. Sz.-Nagy, C. Foias, H. Bercovici, and L. K$\acute{\text{e}}$rchy, {\em{Harmonic Analysis of Operators on Hilbert space}}, Revised and enlarged edition, Universitext, Springer, New York, 2010.
		
                      \bibitem{KM2019} D.\ K.\ Keshari and N.\ Mallick, {\em q-commuting dilation}, Proc. Amer. Math. Soc. \textbf{147} (2019), 655-669.
		
		%\bibitem{MSS} A. Maji, J. Sarkar and T. R. Sankar, {\em Pairs of commuting isometries}, I, Studia Math. 248
		(2019), 171–189.
		\bibitem {GS} M. Gerhold, O. M. Shalit,  {\em Dilations of {$q$}-commuting unitaries}, Int. Math. Res. Not. IMRN, (2022), 1, 63--88.
                     \bibitem{MS} N. Mallick, and K. Sumesh, {\em On a generalization of {A}ndo's dilation theorem}, Acta Sci. Math. (Szeged), 86, (2020), 273--286.

                   %   \bibitem{Opela} D. Op\v{e}la, {\em  A generalization of And\^o’s theorem and Parrott’s example}, Proc. Amer. Math. Soc. {\bf134} (2006), no. 9, 2703--2710.

		%\bibitem{Pal2} S. Pal, {\em Subvarieties of the tetrablock and von Neumann’s inequality}, Indiana Univ. Math. J. 65 (2016), no. 6, 2051-2079.
		
		%\bibitem{Pal} S. Pal. {\em Distinguished varieties in the polydisc and dilation of commuting contractions}, arXiv:2205.00540.
		
                     \bibitem{Parrott} S. Parrott, {\em Unitary dilations for commuting contractions}, Pacific J.~Math., 34(1970), 481-490.
		
		%\bibitem{PO} D. Popovici, {\em A Wold-type decomposition for commuting isometric pairs}, Proc. Amer. Math. Soc. 132 (2004), no. 8, 2303-2314
		
		%\bibitem{DS} D.  Scheinker, {\em Hilbert function spaces and the Nevanlinna-Pick problem on the polydisc}, J. Funct. Anal. 261 (2011), no. 8, 2238–2249.
		
		%\bibitem{Srijan}S. Sarkar, {\em Pairs of commuting pure contractions and isometric dilation}, J. Operator Theory (to appear). arXiv:2105.03051v2 [math.FA].
		
                     \bibitem{Sebestyen} Z. Sebesty\'en, {\em Anticommutant lifting and anticommuting dilation}, Proc. Amer. Math. Soc. {\bf121} (1994), 133-136.
		
		%\bibitem{sauAndo} H. Sau, {\em{And\^{o} dilations for a pair of commuting contractions: two explicit constructions and functional models}}, arXiv:1710.11368 [math.FA].
		
                   \bibitem{sz-nagy}B.  Sz.-Nagy, {\em{Sur les contractions de l'espace de Hilbert}}, Acta Sci. Math. \textbf{15} (1953), 87-92.
		
		%\bibitem{Rudin} W. Rudin, {\em Pairs of inner functions on finite Riemann surfaces}, Trans. Amer. Math. Soc., 140:423434, 1969.
		
                      \bibitem{Schaffer} J.J. Sch\"affer, {\em On unitary dilations of contractions}, Proc. Amer. Math. Soc. 6 (1955), 322. MR 16, 934c.
		
		
		%\bibitem{Taylor} Taylor, J. L., {\em A joint spectrum for several commuting operators}, J. Funct. Anal. 6: 172-191 (1970).
		
		
		%\bibitem{Timko1} E. Timko, {\em On polynomial $n$-tuples of commuting isometries}, J. Operator Theory 77 (2017), 391–420.
		
		
		%\bibitem{Vegulla} P. Vegulla, Geometry of distinguished varieties, Ph.D. thesis, Washington University in St. Louis, 2007, available at https://search.proquest.com/docview/304800270.
		
		%\bibitem{Ves} E. Vesentini, {\em Maximum theorems for vector-valued holomorphic functions}, Rend. Sem. Mat. Fis. Milano 40 (1970), 24-55.
		
		%\bibitem{W} L. Waelbroeck, {\em Le calcule symbolique dans les alg bres commutatives}, Journal Math. Pures Appl., 33 (1954), 147-186.
 
 		
		\bibitem{von} J. von Neumann, {\em Allgemeine Eigenwerttheorie Hermitescher Funktionaloperatoren}, (German) Math. Ann. 102 (1930), no. 1, 49–131. 
		
                      \bibitem{Tomar}N. Tomar, {\em The $q$-commuting contractions and $*$-regular dilations}, J. Math. Anal. Appl. (2025), Paper No. 130214.

		\bibitem{Wold} H. Wold, {\em A Study in the Analysis of Stationary Time Series}, Almqvist \& Wiksell, Stockholm,
		(1954).
		
		
		
		
	\end{thebibliography}
\end{document}